\documentclass[a4paper,10pt]{amsart}

\usepackage{amsmath}
\usepackage{amsfonts}
\usepackage{amssymb}

\usepackage{bbm}
\usepackage{natbib}

\newtheorem{theorem}{Theorem}
\theoremstyle{plain}
\newtheorem{corollary}{Corollary}
\newtheorem{definition}{Definition}

\newtheorem{lemma}{Lemma}
\newtheorem{notation}{Notation}
\newtheorem{proposition}{Proposition}
\newtheorem{remark}{Remark}

\DeclareMathOperator{\argmin}{argmin}

\numberwithin{equation}{section}

\def \N {\mathcal{N}}
\def \R {\mathcal{R}}
\def \L {\mathcal{L}}
\def \F {\mathbb{F}}

\renewcommand{\phi}{\varphi}
\renewcommand{\epsilon}{\varepsilon}

\begin{document}
\title[Finite state mean field games]{Probabilistic approach to finite state mean field games}
\author{Alekos Cecchin}
\address[A. Cecchin and M. Fischer]
{\newline \indent Department of Mathematics ``Tullio Levi Civita''
\newline 
\indent University of Padua \newline
\indent Via Trieste 63, 35121 Padova, Italy
\newline }
\email[A. Cecchin]{alekos.cecchin@math.unipd.it}
\author{Markus Fischer}
\email[M. Fischer]{fischer@math.unipd.it}%
\urladdr{http://www.math.unipd.it/~fischer}

\thanks{The first author is supported by the PhD program in Mathematical Sciences, Department of Mathematics, 
University of Padua (Italy) and Progetto Dottorati - Fondazione Cassa di Risparmio di Padova e Rovigo (CaRiPaRo). The second author acknowledges partial support through the research projects ``Mean Field Games and Nonlinear PDEs'' (CPDA157835) of the University of Padua and ``Nonlinear Partial Differential
Equations: Asymptotic Problems and Mean-Field Games'' of the Fondazione CaRiPaRo. Both authors thank an anonymous Referee for her/his helpful critique and detailed comments and suggestions.}

\date{April 2, 2017; revised December 1, 2017}


\begin{abstract}

We study mean field games and corresponding $N$-player games in continuous time over a finite time horizon where the position of each agent belongs to a finite state space. As opposed to previous works on finite state mean field games, we use a probabilistic representation of the system dynamics in terms of stochastic differential equations driven by Poisson random measures. Under mild assumptions, we prove existence of solutions to the mean field game in relaxed open-loop as well as relaxed feedback controls. Relying on the probabilistic representation and a coupling argument, we show that mean field game solutions provide symmetric $\epsilon_N$-Nash equilibria for the $N$-player game, both in open-loop and in feedback strategies (not relaxed), with $\epsilon_N\leq \frac{\text{constant}}{\sqrt{N}}$. Under stronger assumptions, we also find solutions of the mean field game in ordinary feedback controls and prove uniqueness either in case of a small time horizon or under monotonicity.

\end{abstract}
\maketitle

\section{Introduction}

Mean field games, as independently introduced by \citet{lasrylions07} and by \citet{huangetal06}, represent limit models for symmetric non-zero-sum non-cooperative $N$-player dynamic games with mean field interactions when the number $N$ of players tends to infinity. For an introduction to mean field games see \citet{cardaliaguet13}, \citet{carmonaetal13} and \citet{bensoussanetal13}; the latter two works also deal with optimal control problems of McKean-Vlasov type. There is by now a wealth of works dealing with different classes of mean field games; for a partial overview see \citet{gomesetal15} and the references therein. Here, we restrict attention to a class of finite time horizon problems with continuous time dynamics and fully symmetric cost structure, where the position of each agent belongs to a finite state space.

The relation between the limit model (the mean field game) and the corresponding prelimit models (the $N$-player games) can be understood in two different directions: approximation and convergence. By approximation we mean that a solution of the mean field game allows to construct approximate Nash equilibria for the $N$-player games, where the approximation error is arbitrarily small for $N$ big enough. By convergence we mean that Nash equilibria for the $N$-player games may be expected to converge to a solution of the mean field game as $N$ tends to infinity.

Results in the approximation direction are more common and usually provide the justification for the definition of the mean field game. When the underlying dynamics is of It\^o type without jumps, such results were established by \citet{huangetal06} and, more recently, by for instance \citet{carmonadelarue13}, \citet{carmonalacker15} and \citet{bensoussanetal16}. When the dynamics is driven by generators of L{\'e}vy type, but with the control appearing only in the drift, an approximation result is found in \cite{kolokoltsovetal11}.

Rigorous results on convergence to the mean field game limit in the non stationary case (finite time horizon) are even more recent. While the limits of $N$-player Nash equilibria in stochastic open-loop strategies can be completely characterized (see \citet{lacker16} and \citet{fischer17} for general systems of It{\^o} type), the convergence problem is more difficult for Nash equilibria in Markov feedback strategies with global state information. A breakthrough was achieved by \citet{cardaliaguetetal15}. Their proof of convergence relies on having a regular solution to the so-called master equation. This is a kind of transport equation on the space of probability measures associated with the mean field game; its solution yields a solution to the mean field game for any initial time and initial distribution. If the mean field game is such that its master equation possesses a unique regular solution, then that solution can be used to prove convergence of the costs associated with the $N$-player Nash equilibria, as well as a weak form of convergence of the corresponding feedback strategies. An important ingredient in the proof is a coupling argument similar to the one employed in deriving the propagation of chaos property for uncontrolled mean field systems \citep[cf.][]{sznitman89}. This kind of coupling argument, in which independent copies of the limit process are compared to their prelimit counterparts, is useful also for obtaining approximation results; cf.\ for instance the above cited works by \citet{huangetal06} and \citet{carmonadelarue13}.

In this paper, we focus on games where the position of each agent belongs to a given finite state space $\Sigma := \left\{1,\ldots, d\right\}$. Such games have been studied by \citet{gomesetal13}, and also by \citet{basnaetal14}. Their approach to the problem is based on \mbox{PDE}\,/\,\mbox{ODE} methods and the infinitesimal generator ($Q$ matrix) of the system dynamics; we will return to this shortly.

Here, we adopt a different approach based on a probabilistic representation. We write the dynamics of the $N$-player game as a system of stochastic differential equations driven by independent stationary Poisson random measures with the same intensity measure $\nu$, weakly coupled through the empirical measure of the system states:
\begin{equation}
	X^N_i(t)= \xi^N_i + \int_0^t\int_U f(s, X^N_i(s^-), u, \alpha^N_i(s), \mu^N(s^-)) \N_i^N(ds,du), \quad i=1,\ldots,N,
\label{mfgNi}
\end{equation}
where $\alpha^N_i$ is the control of player $i$ (here in open loop form) with values in a compact set $A$ and $\mu^N(s^-)$ is the empirical measure of the system immediately before time $s$. The dynamics for the one representative player of the mean field limit is analogously written as
\begin{equation}
X(t)= \xi + \int_0^t\int_U f(s, X(s^-), u, \alpha(s), m(s)) \N(ds,du),
\label{mfgi}
\end{equation}
where $\alpha$ is the control and $m:[0,T]\rightarrow\mathcal{P}(\Sigma)$ a deterministic flow of probability measures, which takes the place of $\mu^N$. 

Representations \eqref{mfgNi} and \eqref{mfgi} of the system dynamics allow to obtain approximation results with error bounds of the form $\frac{\text{constant}}{\sqrt{N}}$ for the approximate $N$-player Nash equilibria via the aforementioned coupling argument. This is what we will do here. The probabilistic representation is useful also for the problem of convergence to the mean field limit; see below.

The function $f$ appearing in (\ref{mfgNi}) and (\ref{mfgi}) can be chosen so that the corresponding state processes $X^{N}_{i}$, $X$ have prescribed transition rates when the control and measure variable are held constant. Following an idea of \citet{graham92}, we choose $U\subset \mathbb{R}^{d}$, let the intensity measure $\nu$ be given by $d$ copies of Lebesgue measure on the line (cf.\ (\ref{nu}) below), and set 
\begin{equation} \label{fconti}
	f(t,x,u,a,p) := \sum_{y\in \Sigma} (y-x) \mathbbm{1}_{]0, \lambda(t,x,y,a,p)[} (u_y).
\end{equation}
With this $f$ we have, as $h\rightarrow0$,
\[
	P\left[X(t+h)= y | X(t) = x\right] = \lambda(t,x,y,\alpha,m) \cdot h + o(h)
\]
if $y\neq x$, for any constant control $\alpha$ and probability measure $m$. Thus, $\lambda(t,x,y,\alpha,m)$ is the transition rate from state $x$ to state $y$.

We will consider several types of controls: open-loop, feedback, relaxed open-loop and relaxed feedback. Each player wants to optimize his cost functional over a finite time horizon $T$. The coefficients representing running and terminal costs may depend on the measure variable and are the same for all players.

We first study the mean field game and show that it admits a solution in relaxed controls. The solution of the mean field game can be seen as a fixed point. For a given flow of measures $m(\cdot)$, find a strategy $\alpha_m$ that is optimal and let $X^{\alpha_m,m}$ be the corresponding solution of Eq.~\eqref{mfgi}. Now find $m$ such that $Law(X(t)) = m(t)$ for all $t\in [0,T]$. Under mild hypotheses, we prove existence of solutions in relaxed open-loop controls using the Ky Fan fixed point theorem for point-to-set maps. This is analogous to the existence result obtained by \cite{lacker15} for general dynamics driven by Wiener processes. As there, we will characterize solutions to Eq.~(\ref{mfgi}) through the associated controlled martingale problem. In order to write the dynamics when using a relaxed control, we need to work with relaxed Poisson measures in the sense of \citet{kushnerdupuis01}; also see Appendix~A below. The same assumptions that give existence in relaxed open-loop controls also yield existence of solutions in relaxed feedback controls. Relaxed controls are used only for the limit model.

Then we show that those relaxed mean field game solutions provide $\epsilon_N$-Nash equilibria for the $N$-player game both in ordinary open-loop and ordinary feedback strategies. To this end, we approximate a limiting optimal relaxed control by an ordinary one, using a version of the chattering lemma that also works for feedback controls, at least in our finite setting. The approximating control is then shown to 
provide a symmetric $\epsilon_N$-Nash equilibrium, with $\epsilon_N\leq \frac{\text{constant}}{\sqrt{N}}$, decentralized when considering feedback strategies.
As explained above, our proof relies on the probabilistic representation of the system and a coupling argument.

We also study the problem of finding solutions of the mean field game in ordinary feedback controls. There, we need stronger assumptions in order to guarantee the uniqueness of an optimal feedback control for any fixed $m$ (existence always holds). Moreover, we prove that the feedback mean field game solution is unique either if the time horizon $T$ is small enough or if the cost coefficients satisfy the monotonicity conditions of Lasry and Lions (cf.\ below).

Roughly speaking, we need to assume only the continuity of the rates $\lambda$ in order to have relaxed or relaxed feedback mean field game solutions and to obtain $\epsilon_N$-Nash equilibria for the $N$-player game, both open-loop and feedback. Under stronger assumptions, namely affine dependence of $\lambda$ on the control and strict convexity of the cost, we have uniqueness of the optimal feedback control for any $m$ through the uniqueness of the minimizer of the associated Hamiltonian. Under assumptions similar to these latter, \citet{basnaetal14} study the problem in the framework of non-linear Markov processes and find $\frac{1}{N}$-Nash equilibria for the $N$-player game. In  \citet{gomesetal13}, the transition rates coincide with the control, in analogy with the original works of Lasry and Lions, and $\frac{1}{\sqrt{N}}$-Nash equilibria are obtained. Both these works consider ordinary feedback controls only, hence feedback solutions of the mean field game.

The work by \citet{gomesetal13} also contains a result in the convergence direction. More precisely, convergence of $N$-player Nash equilibria in feedback controls to the mean field limit is established, but only if the time horizon $T$ is sufficiently small. Moreover, the authors prove a result about the uniqueness of feedback mean field game solutions for arbitrary time horizon in case the Lasry-Lions monotonicity conditions hold.

Lastly, let us mention several recent preprints. In \cite{donceletal16}, continuous time mean field games with finite state space and finite action space are studied. The authors prove existence of solutions to the mean field game, corresponding to what we call solutions in relaxed feedback controls. Their prelimit models (the $N$-player games) are different and difficult to compare to ours since they are set in discrete time. The second work we mention is \citet{benazzolietal17}. There, the authors study a class of mean field games with jump diffusion dynamics. An existence result for the mean field game in the spirit of \cite{lacker15} is given. The authors also obtain a convergence result in a special situation where Nash equilibria for the $N$-player games can be found explicitly. In their model, the jump heights are directly (and linearly) controlled, not the jump intensities.

The last two preprints, which appeared nearly simultaneously, after submission of the present paper, concern the convergence problem for finite state mean field games. In \cite{alekosguglielmo}, a joint work of the first author, the convergence of feedback Nash equilibria to solutions of the mean field game is studied following the ideas of \citet{cardaliaguetetal15} sketched above. The Master Equation, which in this case is a first order \mbox{PDE} stated in $\mathcal{P}(\Sigma)$, is employed to obtain convergence of the feedback Nash equilibria, the value functions and a propagation of chaos property for the $N$-player optimal trajectories. Provided that the Master Equation possesses a (unique) classical solution, convergence is established through a coupling argument, which relies on the probabilistic representation of the dynamics introduced here. Existence of a unique classical solution to the Master Equation is verified under the Lasry-Lions monotonicity conditions. In addition, a central limit theorem and a large deviation principle for the $N$-player empirical measure processes are proved. In the independent work by \citet{bayraktarcohen17}, the authors again use the Master Equation in the spirit of \citet{cardaliaguetetal15} to find the same convergence result as in \cite{alekosguglielmo}, but using a slightly different probabilistic representation of the dynamics. They also obtain a central limit theorem for the fluctuations of the $N$-player empirical measure processes.

\subsection*{Structure of the paper}
In Section~2, we introduce the notation and various assumptions to be used in the sequel. Then we describe the $N$-player games as well as the corresponding mean field game, giving the relevant definitions of Nash equilibrium and solution of the mean field game. Relaxed controls (open-loop and feedback) are introduced there as well, while a proper definition of relaxed Poisson measures is given in Appendix~A. All main  assumptions are verified to hold for the natural shape of $f$ in (\ref{fconti}). 

In Section~3, we establish existence of solutions to the mean field game in relaxed open-loop as well as relaxed feedback controls.

In Section~4, we find, under additional assumptions, mean field game solutions in non-relaxed feedback controls by proving the uniqueness of the optimal control for any flow of measures. Moreover, uniqueness of solutions is proved either for small $T$ or under the Lasry-Lions monotonicity conditions.

In Section~5, we first establish a version of the chattering lemma that works also for feedback controls. Then we turn to the construction of approximate Nash equilibria coming from a solution of the mean field game, and derive the error bound mentioned above for feedback as well as open-loop strategies.

Section~6 contains a summary of the main results.

\section{Description of the model}
\label{SectDescription}

\subsection{Notations and assumptions}

Throughout the paper, we fix $\Sigma = \{ 1,\ldots, d\}$ to be the finite state space of any player. Let $T$ be the finite time horizon and $(A, dist)$ be a compact metric space, the space of control values. Let  $U$ be a compact set in $\mathbb{R}^d$ and let $\nu$ be a Radon measure on $U$. Let
\[
	S := \mathcal{P} (\Sigma) = \{ p\in \mathbb{R}^d : p_j \geq 0, \quad j=1,\ldots,d ;\quad p_1 + \ldots + p_d =1 \}
\]
be the space of probability measures on $\Sigma$, which is the probability simplex in $\mathbb{R}^d$.
Let $f: [0,T]\times\Sigma\times U \times A \times S \longrightarrow \{-d,\ldots, d \}$ be a measurable function (the one appearing in the dynamics (\ref{mfgi}) and (\ref{mfgNi})) such that 
	$f(t,x,u,a,p) \in \{1-x,\ldots, d-x \}$.
Let $c:[0,T]\times\Sigma\times A \times S \longrightarrow \mathbb{R}$, $\psi:\Sigma\times S \longrightarrow \mathbb{R}$ be measurable functions, representing the running and the terminal costs, respectively, which will be the same for all players.

We will denote by $\N$ any stationary Poisson random measure on $[0,T]\times U$ with intensity measure $\nu$ on $U$, and by $\N^N=(\N_1^N,\ldots,\N_N^N)$ a vector of $N$ i.i.d.\ stationary Poisson random measures, each with the same law as $\N$. The initial point of the $N$-player game will be represented by $N$ i.i.d.\ random variables $\xi_1,\ldots,\xi_{N}$ with values in $\Sigma$ and common distribution $m_0\in S$, which will be fixed throughout. Similarly, the initial point of the limiting system will be represented by a random variable $\xi$ with law $m_0$.

The state of player $i$ at time $t$ is denoted by $X_i^N(t)$. 
The trajectories of any process $X_i^N$ are assumed to be in $D([0,T],\Sigma)$, which denotes the space of c{\`a}dl{\`a}g functions from $[0,T]$ to $\Sigma$, endowed with the Skorokhod $J_1$-topology.
Let $\mu^N(t):= \frac{1}{N} \sum_{i=1}^N \delta_{X^N_i (t)}$ be the empirical measure of the system of $N$ players. 
In the limiting dynamics, the empirical measure is replaced by a deterministic \emph{flow} of probability measures $m: [0,T]\longrightarrow S$.

The space of measures $S$ can be equipped with any norm in $\mathbb{R}^d$, as they are all equivalent, 
so we choose the Euclidean norm $|p|$. We observe that $S$ is a compact and convex subset of $\mathbb{R}^d$. 
Denote by $\mathcal{C}([0,T],S)$ the space of continuous functions from $[0,T]$ to $S$, endowed with the uniform norm.
The space of flows of probability measures on $S$ will be denoted by $\mathcal{L} \subset \mathcal{C}([0,T],S)$, which will be shown in Subsection~3.1 to be 
\[
	\mathcal{L} := \left\{m:[0,T]\longrightarrow S: |m(t)-m(s)| \leq K|t-s| , \quad m(0) = m_0 \right\}
\]
where the constant is given by $K:=2\nu(U)\sqrt{d}$.

We will study several types of controls. Pathwise existence and uniqueness of solutions to the controlled dynamics (\ref{mfgi}), with trajectories that remain in $\Sigma$, is guaranteed by the following Lipschitz condition:
\begin{equation}
	\int_U |f(s,x, u, a, p) - f(s,x, u, a, p)|\nu(du) \leq K_1 |x-y| 
\label{lip}
\end{equation}
for every $ x,y\in \Sigma, s\in [0,T], a\in A$ and $p\in S$,
where $K_1$ is a constant.
The above condition is always satisfied in our model since $|x-y|\geq 1 $ for each $x\neq y \in \Sigma$ and 
$\int_U |f(s,x, u, a, p)| \leq  \nu(U) d$; thus we may take $K_1= 2 \nu(U) d$.

Let us summarize here the various sets of assumptions we will make use of:

\begin{itemize}
	\item[\textbf{(A)}] The function $\tilde{f}:[0,T] \times \Sigma  \times A\times S \longrightarrow L^1(\nu)$ defined by
	$\tilde{f}(t,x, a,p):=  f(t, x,\cdot ,a,p)  \in L^1(\nu)$ is continuous in $t,a,p$ (uniformly, and is bounded), that is, there exists a function $w_f$ such that $\lim_{h\rightarrow0} w_f(h)=0$ and
	\begin{equation}
	\int_U |f(t, x, u, a, p)- f(s, x,u,b, q)|\nu(du) \leq w_f(|t-s| + dist(a,b) + |p-q|)
	\label{continua}
	\end{equation}
	for every $t,s\in [0,T]$, $x\in\Sigma$, $a,b \in A$, $p, q \in S$;
	
	\item[\textbf{(A')}] Assumption (A) holds and $\tilde{f}$ is Lipschitz in $p\in S$:
	\begin{equation}
		\int_U |f(t, x, u, a, p)- f(t, y,u,a, q)|\nu(du) \leq K_1 (|x-y| + |p-q|);
	\label{lipA}
	\end{equation}
	
	\item[\textbf{(A'')}] Assumption (A') holds and$\tilde{f}$ is Lipschitz also in $a\in A$:
	\begin{equation}
	\int_U |f(t, x, u, a, p)- f(t, y,u,b, q)|\nu(du) \leq K_1 (|x-y| + |p-q| + dist(a,b));
	\label{lipC}
	\end{equation}
	
	\item[\textbf{(B)}] The running cost $c$ is continuous (and bounded) in $t, x, a, p$ and the terminal cost is continuous (and bounded) in $x,p$;
		
	\item[\textbf{(B')}] Assumption (B) holds and the costs  $c$ and $\psi$ are Lipschitz in $p$:
	\begin{equation}
	|c(t,x,a,p) -c(t,y,a,q)| + |\psi(x,p) -\psi(y,q)|  \leq K_2 (|x-y| + |p-q|); 
	\label{lipB}
	\end{equation}
	
	\item[\textbf{(B'')}] Assumption (B') holds and the running cost $c$ is Lipschitz also in $a$:
	\begin{equation}
	|c(t,x,a,p) -c(t,y,b,q)| + |\psi(x,p) -\psi(y,q)|\leq K_2 [|x-y| + dist(a,b) + |p-q|] . 
	\label{lipD}
	\end{equation}
\end{itemize}

The above assumptions will be used in Sections 3 and 5 to find solutions of the mean field game and then approximate Nash equilibria for the $N$-player game, both in open-loop and in feedback form.

Our last assumption will be more implicit.
We identify the set of functions $g:\Sigma\longrightarrow\mathbb{R}$ with $\mathbb{R}^d$ and observe that any $g$ is bounded and Lipschitz. 
For any $x\in\Sigma$, $0\leq t\leq T$, $a\in A$, $p\in S$ and $g\in\mathbb{R}^d$ define the \emph{generator} 
\begin{equation} \label{gen}
	\Lambda^{a,p}_t g(x):= \int_U [g(x + f(t,x, u, a, p)) -g(x)] \nu(du)
\end{equation}
and the \emph{pre-Hamiltonian}
\begin{equation}
	H(t,x,a,p,g) := \Lambda^{a,p}_t g(x) + c(t,x,a,p).
\label{H}
\end{equation}
In order to obtain existence and uniqueness of feedback mean field game solutions, in Section~4, we will make the additional hypothesis:
\begin{itemize}
	\item[\textbf{(C)}] For any $t$, $x$, $p$ and $g$ there exists a unique $a^\ast = a^\ast(t,x,p,g)$ minimizer of $H(t,x,a,p, g)$ in $A$;
\end{itemize}
We observe that for any fixed $p$ and $g$ the function $a^\ast(t,x)$ is measurable, thanks to Theorem D.5 in \citet{hernandezlermalasserre96}. We remark also that the limiting dynamics (\ref{mfgi}) always admits a pathwise unique solution thanks to (\ref{lip}).

\subsection{N-player game}

In the prelimit, we consider a system of $N$ symmetric players governed by the dynamics
\begin{equation}
	X^N_i(t)= \xi^N_i + \int_0^t\int_U f(s, X^N_i(s^-), u, \alpha^N_i(s), \mu^N(s^-)) \N_i^N(ds,du) \qquad i=1,\ldots,N,
\label{Nopen}
\end{equation}
where $X^N=(X^N_1, \ldots, X^N_N)$ and $\mu^N(t):= \frac{1}{N} \sum_{i=1}^N \delta_{X^N_i (t)}$. Here, the controls $\alpha_i^N$ are in open-loop form. Let us specify the controls to be used in the $N$-player game.

\begin{definition}
Define the set of \emph{strategy vectors} as
\[
	\mathcal{A}^N := \left\{((\Omega, \mathcal{F}, P; \mathbb{F}) , \alpha^N, \xi^N, \N^N) \right\}
\]
where $(\Omega, \mathcal{F}, P; \mathbb{F})$ is a filtered probability space, 
$\xi^N := (\xi^N_1,\ldots, \xi^N_N)$ is a vector of $N$ i.i.d.\ $\mathcal{F}_0$-measurable random variables with law $m_0$, the initial points, 
$\N^N = (\N_1^N,\ldots,\N_N^N)$ is a vector of $N$ i.i.d.\ stationary Poisson random measures with respect to the filtration 
$\mathbb{F}=(\mathcal{F}_t)_{t\in [0,T]}$ with intensity measure $\nu$ on $U$, $\mathcal{F}_T=\mathcal{F}$, and $\alpha^N= (\alpha^N_1,\ldots, \alpha^N_N)$ is a vector of $A$-valued $\mathbb{F}$-predictable processes $\alpha^N_i$. We will often write 
$\alpha^N\in\mathcal{A}^N$ to indicate the process $\alpha^N$.

Define the set of \emph{feedback strategy vectors} as 
\[
	\mathbb{A}^N:=\left\{((\Omega, \mathcal{F}, P; \mathbb{F}) , \gamma^N, \xi^N, \N^N) \right\}
\]
where $\gamma^N=(\gamma^N_1,\ldots, \gamma^N_N) :[0,T]\times\Sigma^N \rightarrow A^N$  is measurable and the filtered probability space and the  $\xi^N$ and $\N^N$ are as above. We will often write 
$\gamma^N\in\mathbb{A}^N$ to indicate the function $\gamma^N$.
\label{Ncontrols}
\end{definition}

We observe that the above definition of feedback strategy vector is not standard, as it is given together with the probability space and the noise. 
We give such a definition because in this way any strategy gives a unique pathwise solution to dynamics (\ref{Nopen}). Indeed, provided that $\tilde{f}$ is Lipschitz in $p$, we have pathwise existence and uniqueness of solutions to the system (\ref{Nopen}), for any $\alpha^N = (\alpha^N_1,\ldots, \alpha^N_N)  \in\mathcal{A}^N$. 
	
Given a feedback strategy vector $\gamma^N =(\gamma^N_1,\ldots, \gamma^N_N) \in \mathbb{A}^N$, equation (\ref{Nopen}) is written as 
\[
	X^N_i(t)= \xi^N_i + \int_0^t\int_U f(s, X^N_i(s^-), u, \gamma^N_i(s, X^N(s^-)), \mu^N(s^-)) \N^N_i(ds,du)  
\]
for each $i=1,\ldots,N$.
The same assumption as above provides existence and uniqueness of solutions $X^N_i$ to this equation, so we can define the related open-loop control $\alpha^N[\gamma^N]$ by 
\[
	\alpha^N[\gamma^N]_i(s) = \gamma^N_i(s,X^N(s^-)) .
\]
In view of Definition \ref{Ncontrols}, the open-loop control $\alpha^N[\gamma^N]$ has to be given together with a filtered probability space, a vector of initial conditions  and a vector of Poisson random measures, which we impose to be the same as those given with the feedback control $\gamma^N$.

Next, we define the object of the minimization. Let  $\alpha^N = (\alpha^N_1,\ldots, \alpha^N_N)\in \mathcal{A}^N$ be a strategy vector and $X^N = (X^N_i, \ldots, X^N_N)$ be the solution to dynamics (\ref{Nopen}). For $i=1,\ldots,N$ set
\[
	J^N_i(\alpha^N) := E\left[\int_0^T c(t,X_i^N(t), \alpha_i^N(t), \mu^N(t))dt + \Psi (X_i^N(T),\mu^N(T))\right].
\]
Define also $J^N_i(\gamma^N):= J^N_i(\alpha^N[\gamma^N])$ for any $\gamma^N\in \mathbb{A}^N$.
 
We look for \emph{approximate} Nash equilibria for the $N$-player game. So let us define what are the perturbed strategy vectors we consider.

\begin{notation}
Let $\beta$ be an $A$-valued $\mathbb{F}$-predictable process. For a strategy vector $\alpha^N = (\alpha^N_1,\ldots, \alpha^N_N)$ 
 in $\mathcal{A}^N$ denote by
$[\alpha^{N,-i}; \beta]$ the  strategy vector such that 
\[
[\alpha^{N,-i}; \beta]_j =
\begin{cases}
 \alpha^N_j & j\neq i\\
  \beta & j=i.
\end{cases}
\]
For a feedback strategy vector $\gamma^N= (\gamma^N_1,\ldots, \gamma^N_N)\in \mathbb{A}^N$, let $\widetilde{X}^N$ be the solution to
\[
	\begin{cases}
	\widetilde{X}^N_i(t)= &\xi_i^N + \int_0^t\int_U  f(s, \widetilde{X}^N_i(s^-), u, \beta(s), \widetilde{\mu}_N(s^-)) \N_i^N(ds,du)\\
	\widetilde{X}^N_j(t)=& \xi_j^N + \int_0^t\int_U  f(s, \widetilde{X}^N_j(s^-), u, \gamma_j^N(s, \widetilde{X}^N(s^-)), \widetilde{\mu}_N(s^-)) \N_j^N(ds,du) \\
	& \mbox{if } j\neq i.
\end{cases}
\]
Denote then by $[\gamma^{N,-i}; \beta]\in\mathcal{A}^N$ the strategy vector such that
\[
	[\gamma^{N,-i}; \beta]_j (t) =
\begin{cases}
 \gamma^N_j (t, \widetilde{X}^N_j(t^-)) & j\neq i\\
  \beta(t) & j=i.
\end{cases}
\]
\label{n1}
\end{notation}

\begin{definition}
Let $\epsilon>0$. A strategy vector $\alpha_N$ is said to be an \emph{$\epsilon$-Nash equilibrium} if for each $i=1,\ldots,N$ 
\[
	J^N_i(\alpha^N) \leq J_N^i([\alpha^{N,-i}; \beta]) + \epsilon
\]
for every $\beta$ such that $[\alpha^{N,-i}; \beta]$ is a strategy vector.

A vector $\gamma^N \in \mathbb{A}^N$ is called a \emph{feedback} $\epsilon$-Nash equilibrium if 
\[
	J^N_i(\gamma^N) \leq  J_N^i([\gamma^{N,-i}; \beta]) + \epsilon
\]
for every $\beta$ such that $[\gamma^{N,-i}; \beta]$ is a strategy vector.

\label{eqnash}
\end{definition}

We remark that the above definition of feedback $\epsilon$-Nash equilibrium is not standard. 
Indeed, the perturbed strategy vector $[\gamma^{N,-i}; \beta]$ is usually required to be in feedback form. In our definition, a slightly more restrictive (or stronger) condition is used since the perturbing strategy $\beta$ is allowed to be in open-loop form. As a consequence, the approximation result of Section~5 will be slightly stronger than with the standard definition.

\subsection{Mean field game}

The \emph{mean field} limiting system  consists of a single player whose state evolves according to the dynamics
\begin{equation}
	X(t) = \xi +\int_0^t \int_U f(s,X(s^-), u, \alpha(s), m(s)) \N(ds,du),\quad t\in [0,T].
\label{mfg}
\end{equation}
Here the empirical measure appearing in (\ref{Nopen}) is replaced by a deterministic \emph{flow} of probability measures $m: [0,T]\longrightarrow S$.

\begin{definition}
The set of \emph{open-loop controls} is the set  
\[
	\mathcal{A}:= \left\{((\Omega, \mathcal{F}, P; \mathbb{F}) , \alpha, \xi, \N) \right\}
\]
where  $(\Omega, \mathcal{F}, P; \mathbb{F})$ is a filtered probability space, $\xi$ is an $\mathcal{F}_0$-measurable random variable with law $m_0$, the initial condition, $\N$ is a stationary Poisson random measure with respect to the filtration 
$\mathbb{F}=(\mathcal{F}_t)_{t\in [0,T]}$ with intensity measure $\nu$ on $U$, $\mathcal{F}_T=\mathcal{F}$, and $\alpha$ is an $A$-valued $\mathbb{F}$-predictable process. We will often write 
$\alpha\in\mathcal{A}$ to indicate the process $\alpha$.

Define the set of \emph{feedback controls} as 
\[
	\mathbb{A}:=\left\{((\Omega, \mathcal{F}, P; \mathbb{F}) , \gamma, \xi, \N) \right\}
\]
where $\gamma:[0,T]\times\Sigma \rightarrow A$  is measurable and the filtered probability space, the initial condition and the Poisson random measure $\N$ are as above. We will often write 
$\gamma\in\mathbb{A}$ to indicate the function $\gamma$.
\label{controls}
\end{definition}
	
We remark that the feedback control is given with the probability space and the noise, in analogy with Definition \ref{Ncontrols} for the prelimit system.

Thanks to the Lipschitz condition (\ref{lip}), the limiting dynamics is well defined. More precisely, given any open-loop control $((\Omega, \mathcal{F}, P; \mathbb{F}) , \alpha, \xi, \N) \in\mathcal{A}$ and flow of measures $m\in\mathcal{L}$, there exists a pathwise unique solution $X$ of Eq.~(\ref{mfg}), which we will denote by $X_{\alpha, m}$. Similarly, given any feedback control $((\Omega, \mathcal{F}, P; \mathbb{F}) , \gamma, \xi, \N) \in\mathbb{A}$ and flow of measures $m\in\mathcal{L}$, there exists a pathwise unique process $X = X_{\gamma,m}$ solving
\begin{equation}
	X(t) = \xi +\int_0^t \int_U f(s,X(s^-), u,\gamma(s,X(s^-)) , m(s)) \N(ds,du),\quad t\in [0,T].
\label{Xfeed}
\end{equation}
The corresponding open-loop control is then defined as
\begin{equation}
	\alpha^{\gamma} (t) := \gamma(t,X_{\gamma,m}(t^-)).
\label{algam}
\end{equation}

In view of Definition \ref{controls}, the open-loop control $\alpha^{\gamma}$ has to be given together with a filtered probability space, an initial condition and a Poisson random measure, which we impose to be the same as those given with the feedback control $\gamma$.

We define the object of the minimization for the mean field game. For any $\alpha \in \mathcal{A}$ and $m\in \mathcal{L}$ set
\[
	J(\alpha, m) := E\left[ \int_0^T c(s,X_{\alpha, m} (s), \alpha(s), m(s))ds + \psi(X_{\alpha, m} (T), m(T)) \right].
\]
Define also $J(\gamma,m):= J(\alpha^\gamma,m)$ for any $\gamma\in \mathbb{A}$.

The notion of solution for the limiting mean field game, which will provide approximate Nash equilibria for the $N$-player game, is the following.

\begin{definition}
An \emph{open-loop solution of the mean field game (\ref{mfg})} is a triple
\[
	\left( \left((\Omega, \mathcal{F}, P; \F), \alpha, \xi, \N \right), m, X  \right)
\]
such that 
\begin{enumerate}
	\item $\left((\Omega, \mathcal{F}, P; \F), \alpha, \xi, \N \right) \in \mathcal{A}$, $m\in\mathcal{L}$,
	  $(X(t))_{t\in [0,T]}$ is adapted to the filtration $\mathbb{F}$ and $X = X_{\alpha, m}$;

	\item \emph{Optimality}: $J(\alpha, m) \leq J(\beta, m)$ for every $\beta \in \mathcal{A}$;

	\item \emph{Mean Field Condition}: $Law (X(t)) = m(t)$ for every $t\in [0,T]$.
\end{enumerate}

We say that $\left( \left((\Omega, \mathcal{F}, P; \F), \gamma, \xi, \N \right), m, X  \right)$ is a \emph{feedback solution of the mean field game} if $\gamma\in\mathbb{A}$ and 
$\left( \left((\Omega, \mathcal{F}, P; \F), \alpha^\gamma, \xi, \N \right), m, X  \right)$ is an open-loop solution of the mean field game, where $\alpha^\gamma$ is defined in (\ref{algam}).

In our writing, we will often drop the filtered probability space and the Poisson random measure from the notation.

\label{sol}
\end{definition}

In condition (3) of the above definition, $Law(X(t)):= P\circ X(t)^{-1}$ as usual. Let us denote by $Flow(X):[0,T]\longrightarrow S$ the \emph{flow} of the process $X$, that is, $Flow(X).t:= Law(X(t))$. Then the mean field condition can be written as
$Flow(X) =m$.

\subsection{Relaxed controls}

The space $\mathcal{A}$ is not itself compact. In order to always have convergence along subsequences, we need to enlarge the space of controls, considering relaxed controls and related relaxed Poisson measures. They are used only for the limiting system.

\begin{definition}
A \emph{deterministic relaxed control} is a measure $\rho$ on the Borel sets $\mathcal{B}([0,T]  \times A)$ such that
\begin{equation}
\rho( [0,t[ \times A) = \rho( [0,t] \times A) =t \quad \forall t\in [0,T].
\label{relcon}
\end{equation}
The space of deterministic relaxed controls will be denoted by $\mathcal{D}$.
\label{defrel}
\end{definition}

Given $\rho\in\mathcal{D}$, the time derivative exists for Lebesgue-almost every $t\in (0,T]$; it is the probability measure $\rho_t$ on $A$ given by
\begin{equation}
\rho_t(E) := \lim_{\delta\rightarrow0} \frac{\rho( [t-\delta,t] \times E)}{\delta},\quad E\in\mathcal{B}(A).
\label{defder}
\end{equation}
As a consequence, $\rho$ can be factorized according to
\begin{equation}
\rho(dt,da) = \rho_t(da)dt.
\end{equation}

The space $\mathcal{D}$ is endowed with the topology of weak convergence of measures, i.e. $\rho_n\rightarrow\rho$ if and only if 
\begin{equation}
\int_0^T \int_A\varphi(s,a)\rho_n(ds,da) \quad \longrightarrow \quad \int_0^T \int_A\varphi(s,a)\rho(ds,da)
\label{conti}
\end{equation}
for every continuous $\varphi$ on $[0,T]\times A$. Moreover there exists a metric  which makes $\mathcal{D}$ a compact metric space \citep[for instance,][]{kushnerdupuis01}. 
\begin{definition}
The space of \emph{(stochastic) relaxed controls} is 
\[
	\mathcal{R}= \left\{((\Omega, \mathcal{F}, P; \mathbb{F}), \rho, \xi, \N) \right\}
\]
where  $(\Omega, \mathcal{F}, P; \mathbb{F})$ is a filtered probability space, $\rho$ is a $\mathcal{D}$-valued random variable such that $\rho( [0,\cdot] \times E)$ is $\mathbb{F}$-adapted for every $E\in \mathcal{B}(A)$, and $\N$ is a stationary Poisson random measure with respect to the filtration $\mathbb{F}$ with intensity measure $\nu$ on $U$.
We will often write $\rho\in\mathcal{R}$ to denote the process $\rho$.

The space of  \emph{relaxed feedback controls} is the set
\[
	\widehat{\mathbb{A}} :=  \left\{((\Omega, \mathcal{F}, P; \mathbb{F}), \widehat{\gamma}, \xi, \N) \right\}
\]
where $\widehat{\gamma}: [0,T]\times \Sigma \longrightarrow \mathcal{P}(A)$ is measurable, $\mathcal{P}(A)$ is endowed with the topology of weak convergence, and the filtered probability space, the initial condition and the Poisson random measure are as above. We will often write $\widehat{\gamma}\in\widehat{\mathbb{A}} $ to denote the process $\widehat{\gamma}$.
\label{relcontrols}
\end{definition}

The relaxed feedback control is given with the probability space and the noise, in analogy with Definition \ref{controls}. 
Because of (\ref{defder}), the derivative $(\rho_t(E))_{0\leq t\leq T}$ is an $\mathbb{F}$-predictable process for any $E\in\mathcal{B}(A)$.
An ordinary open-loop control $\alpha\in\mathcal{A}$ can be viewed as a relaxed control $\rho^{\alpha}\in\mathcal{R}$ in which the derivative in time is a Dirac measure:
\[
	\rho^{\alpha}([0,t]\times E)=\int_0^t \rho^{\alpha}_s (E) ds = \int_0^t \delta_{\alpha(s)}(E) ds .
\]

We also have to introduce the corresponding \emph{relaxed Poisson measure} in order to have well-defined dynamics. This will be done properly in Appendix~A. Given any $\rho\in\mathcal{R}$, Borel sets $U_0\subseteq U$, $A_0\subseteq A$, the relaxed Poisson measure $\N_\rho$ related to the relaxed control $\rho$ has the property that the processes
\begin{equation}
	\N_\rho(t,U_0,A_0)- \nu(U_0)\rho([0,t]\times A_0)
\label{mart}
\end{equation}
are $\mathbb{F}$-martingales, and are orthogonal for disjoint $U_0\times A_0$. This martingale property and the fact that $\N_\rho$ is a counting measure valued process define the distribution of $\N_\rho$ and the joint law of $(\N_\rho, \rho, \xi, \N)$ uniquely (see Appendix A).
The martingale property (\ref{mart}) also implies that the process
\begin{equation}
	\int_0^t\int_U\int_A \varphi(s,u,a) \N_\rho(ds,du,da) - \int_0^t\int_U\int_A \varphi(s,u,a) \nu(du)\rho_s(da)ds
\label{M}
\end{equation}
is an $\mathbb{F}$-martingale, for any bounded and measurable $\varphi$.
For an ordinary control $\alpha\in\mathcal{A}$ (or the relaxed control it induces), the corresponding relaxed Poisson measure is explicitly given by
\begin{equation}
	\N_\alpha(t,U_0,A_0) := \int_0^t \int_{U_0}\mathbbm{1}_{A_0}(\alpha(s))  \N(ds,du).
\label{Nal}
\end{equation}

The stochastic differential equation (\ref{mfg}) in this more general framework with a relaxed Poisson measure is written as 
\begin{equation}
	X(t) = \xi + \int_0^t\int_U\int_A  f(s,X(s^-), u, a, m(s)) \N_\rho(ds,du,da)
\label{mfgr}
\end{equation}
for any relaxed control $\rho\in\mathcal{R}$ and $m\in\mathcal{L}$.

Given a  relaxed feedback control $\widehat{\gamma}\in \widehat{\mathbb{A}}$ and a process $X$, define the corresponding relaxed open-loop control through
\begin{equation}
	\rho^{\widehat{\gamma},X}(dt,da) := [\widehat{\gamma}(t,X(t^-))](da)dt.
\label{rg}
\end{equation}
Let $\mathcal{N}_{\rho^{\widehat{\gamma},X}}$ be the relaxed Poisson measure corresponding to $\rho^{\widehat{\gamma},X}$. Equation (\ref{mfgr}) then becomes
\begin{equation}
	X(t) = \xi + \int_0^t\int_U\int_A  f(s,X(s^-), u, a, m(s)) \N_{\rho^{\widehat{\gamma},X}}(ds,du,da),
\label{mfgrf}
\end{equation}
where the solution process $X$ appears also in the relaxed Poisson measure. 

The proof of the following lemma is given in Appendix A.1.
\begin{lemma}
For any $m\in\mathcal{L}$ and $\rho\in\mathcal{R}$, respectively $\widehat{\gamma}\in \widehat{\mathbb{A}}$, there exists a pathwise unique  solution to the stochastic differential equation (\ref{mfgr}), respectively (\ref{mfgrf}). 
\label{lemmaexi}
\end{lemma}

The solutions to (\ref{mfgr}) and (\ref{mfgrf}) will be denoted by $X_{\rho,m}$ and $X_{\widehat{\gamma},m}$ respectively. For 
$\widehat{\gamma}\in \widehat{\mathbb{A}}$, let $\rho^{\widehat{\gamma}}$ denote the corresponding relaxed control defined by (\ref{rg}), that is, $\rho^{\widehat{\gamma}}$ is the relaxed open-loop control such that
\begin{equation} \label{rgg}
	\rho_t^{\widehat{\gamma}}(da) := [\widehat{\gamma}(t,X_{\widehat{\gamma},m}(t^-))](da).
\end{equation}
In view of Definition~\ref{relcontrols}, the relaxed open-loop control $\rho^{\widehat{\gamma}}$ has to be given together with a filtered probability space, an initial condition and a Poisson random measure, which we impose to be the same as those coming with the relaxed feedback control $\widehat{\gamma}$.

Let $\rho\in\mathcal{R}$ and $m\in\mathcal{L}$. Let $X=X_{\rho,m}$. Thanks to the martingale property (\ref{M}), we obtain that the process
\begin{align}
M_g^X(t)&=  g(X(t))-g(X(0)) \label{marg}\\
&- \int_0^t\int_U\int_A \left[ g(X(s) + f(s,X(s), u, a, m(s))) -g(X(s))\right] \nu(du)\rho_s(da)ds
\nonumber
\end{align}
is an $\mathbb{F}$-martingale, for any $g\in\mathbb{R}^d$. This yields the Dynkin formula 
\begin{align}
	&E[g(X(t))] - E[g(\xi)]
\label{dynrel} \\ 
	&= E\int_0^t\int_U\int_A \left[ g(X(s) + f(s,X(s), u, a, m(s))) -g(X(s))\right] \nu(du)\rho_s(da)ds.
\nonumber
\end{align}

The cost to be minimized is 
\begin{equation}
	J(\rho, m) := E\left[ \int_0^T \int_A c(s,X_{\rho, m} (s), a, m(s))\rho_s(da)ds + \psi(X_{\rho, m} (T), m(T)) \right].
\label{JR}
\end{equation}
Define also $J(\widehat{\gamma},m):=J(\rho^{\widehat{\gamma}}, m) $ for $\widehat{\gamma}\in \widehat{\mathbb{A}}$. 
The definitions of \emph{relaxed solution} of the mean field game (\ref{mfgr}) and \emph{relaxed feedback solution} are analogous to Definition \ref{sol}, where ordinary controls are replaced by relaxed controls.

\begin{definition}
A \emph{relaxed solution of the mean field game (\ref{mfg})} is a triple
\[
	\left( \left((\Omega, \mathcal{F}, P; \F), \rho, \xi, \N \right), m, X  \right)
\]
such that 
\begin{enumerate}
	\item $\left((\Omega, \mathcal{F}, P; \F), \rho, \xi, \N \right) \in \mathcal{R}$, $m\in\mathcal{L}$,
	  $(X(t))_{t\in [0,T]}$ is adapted to the filtration $\mathbb{F}$ and $X = X_{\rho, m}$;

	\item \emph{Optimality}: $J(\rho, m) \leq J(\sigma, m)$ for every $\sigma \in \mathcal{R}$;

	\item \emph{Mean Field Condition}: $Law (X(t)) = m(t)$ for every $t\in [0,T]$.
\end{enumerate}

We say that $\left( \left((\Omega, \mathcal{F}, P; \F), \widehat{\gamma}, \xi, \N \right), m, X  \right)$ 
is a \emph{relaxed feedback solution of the mean field game} if $\widehat{\gamma}\in\widehat{\mathbb{A}}$ and 
$\left( \left((\Omega, \mathcal{F}, P; \F), \rho^{\widehat{\gamma}}, \xi, \N \right), m, X  \right)$ is a relaxed solution of the mean field game, where $\rho^{\widehat{\gamma}}$ is defined in (\ref{rgg}).

In our writing, we will often drop the filtered probability space and the Poisson random measure from the notation.

\label{solrel}
\end{definition}

In Section~3.2 we will show the existence of relaxed mean field game solutions via a fixed point argument, while existence of a relaxed feedback mean field game solution is established in Section~3.3. 

We will use the characterization of solutions to \eqref{mfgr} via the controlled martingale problem. The proof of the following lemma is omitted; it can be derived by mimicking the one of Theorem 2.8.1 in \citet[p.\,42]{kushner90}.

\begin{lemma}
Let $\left((\Omega', \mathcal{F}', P'; \F'), \rho, \xi, \N \right) \in \mathcal{R}$ and $m\in\mathcal{L}$. Then $X$ solves equation (\ref{mfgr}) in distribution if and only if the process $M_g^X(t)$ defined in (\ref{marg})
is an $\mathbb{F}$-martingale for any $g\in\mathbb{R}^d$.  The underlying filtered probability space can always be assumed to be
$D([0,T],\Sigma)\times\Omega$, where $\Omega$ is the canonical space for $(\N_\rho, \rho, \xi, \N)$ defined in Appendix A, $\mathbb{F}$ the canonical filtration, and $X$ is the canonical process. 

The martingale property holds if and only if 
\begin{equation}
	E\left[h(X(t_i); i\leq j ) (M_g^X(t+s)-M_g^X(t))\right] =0
\label{prmar}
\end{equation}
for every $h:\Sigma^j \rightarrow \mathbb{R}$ and every choice of $j$, $t$, $s$, $t_i$, $i=1,\ldots,j$ such that $0\leq t_i \leq t\leq t+s$.

\label{conmar}
\end{lemma}

In Section~4, under additional assumptions, we will prove existence of feedback mean field game solutions (not relaxed); such solutions will be shown to be unique either if the time horizon is small or if the Lasry-Lions monotonicity assumptions apply.

\subsection{Example}

 We show how our assumptions are satisfied for a natural shape of the function $f$ for which, when considering $\alpha$ and $m$ constants, the transition rates of the Markov chain $X$ solution of the dynamics (\ref{mfg}) appear explicitly. 
Consider then $f$ defined by 
\begin{equation}
f(t,x,u) := \sum_{y\in \Sigma} (y-x) \mathbbm{1}_{]0, \lambda(t,x,y)[} (u_y)
\label{effe}
\end{equation}
and  the intensity measure $\nu$ on $U\in \mathcal{B}(\mathbb{R}^d)$ defined by
\begin{equation}
\nu(E) := \sum_{y=1}^d \ell (E \cap U_y),\quad E\in \mathcal{B}(U),
\label{nu}
\end{equation}
where $U_y :=\left\{u\in U : u_z =0 \quad \forall z\neq y\right\}$, which is viewed as a subset of $\mathbb{R}$, 
and $\ell$ is the Lebesgue measure on $\mathbb{R}$.

The function $\lambda$ appearing in (\ref{effe}) yields the transition rates of the Markov chain $X$ solution of (\ref{mfg}), that is, for $x\neq y$, as $h\rightarrow0$,
\begin{equation}
P\left[X(t+h)= y | X(t) =x\right] = \lambda(t,x,y) \cdot h + o(h).
\label{rate}
\end{equation}
Moreover, the measure $\nu$ defined in (\ref{nu}) has the property that
\begin{equation}
\int_U \varphi(u_1,\ldots, u_d) \nu(du) = \sum_{y=1}^d \int_{U_y} \varphi (0,\ldots,u_y,\ldots,0) du_y
\label{intphi}
\end{equation}
for any bounded and measurable $\varphi:\mathbb{R}^d \longrightarrow \mathbb{R}$. In particular,
\begin{equation}
\int_U \sum_{y=1}^d \varphi_y (u_y) \nu (du) = \sum_{y=1}^d \int_{U_y} \varphi_y (u_y) du_y
\label{int}
\end{equation}
for any function $\varphi_y :\mathbb{R} \longrightarrow \mathbb{R}$ such that $\varphi_y(0)=0$, $y\in\Sigma$. 

If we want $f$ to depend also on a control and a flow of measures, we may consider the rate $\lambda$ to depend also on $a\in A$ and $p\in S$, so that (\ref{effe}) is rewritten as 
\begin{equation}
	f(t,x,u,a,p) := \sum_{y\in \Sigma} (y-x) \mathbbm{1}_{]0, \lambda(t,x,y,a,p)[} (u_y).
\label{fcont}
\end{equation}
We also assume that $\lambda$ is bounded by a constant $M$ (which holds a posteriori by the assumptions of the next lemma) and 
$U:= [0,M]^d$.
With this $f$, (\ref{rate}) becomes
\begin{equation}
	P\left[X(t+h)= y | X(t) =x\right] = E_{t,x}[\lambda(t,x,y,\alpha(t),m(t))] \cdot h + o(h),
\label{rate2}
\end{equation}
where $X=X_{\alpha,m}$ is the solution of (\ref{mfg}) under the control $\alpha\in \mathcal{A}$ and flow of measures $m\in\mathcal{L}$ and $E_{t,x}$ denotes expectation with respect to the conditional probability $P[\cdot|X(t)=x]$ provided $P(X(t)=x)>0$. In particular, if $\alpha(t)= \gamma(X(t))$, then the transition rate is $\lambda(t,x,y,\gamma(x),m(t))$. A proof of (\ref{rate}), (\ref{intphi}) and (\ref{rate2}) can be found in \cite{turchi15}, where the examples (\ref{effe}) and (\ref{fcont}) were treated.

Let us check whether our assumptions on the model are satisfied for the above choice of $f$ and $\nu$.

\begin{lemma}
Let $f$ be defined by (\ref{fcont}) and $\nu$  by (\ref{nu}). 
\begin{itemize}
	\item If the rate $\lambda$ appearing in (\ref{fcont}) is  continuous in $t,a$ and $p$, then  (A) holds;
	\item If in addition $\lambda$ is Lipschitz in $p$, then (A') holds;
	\item If in addition $\lambda$ is Lipschitz also in $a$, then (A'') holds.
\end{itemize}
\label{lemf1}
\end{lemma}

\begin{proof}
Let $t,s\in [0,T]$, $a,b\in A$, $p,q\in S$ and fix $x\in \Sigma$. Then
\begin{align*}
	&\int_U |f(t, x, u, a, p)- f(s, x,u,b, q)|\nu(du) \\
	&=\int_U \left|\sum_{y\in \Sigma} (y-x) \left[\mathbbm{1}_{]0, \lambda(t,x,y,a,p)[} (u_y) - \mathbbm{1}_{]0, \lambda(s,x,y,b,q)[} (u_y)\right]\right| \nu(du)\\
	&\leq \int_U \sum_{y\neq x} |y-x| \left|\mathbbm{1}_{]0, \lambda(t,x,y,a,p)[} (u_y) - \mathbbm{1}_{]0, \lambda(s,x,y,b,q)[} (u_y)\right| \nu(du)\\
	&\leq 2d \int_U \sum_{y\neq x} \left|\mathbbm{1}_{]0, \lambda(t,x,y,a,p)[} (u_y) - \mathbbm{1}_{]0, \lambda(s,x,y,b,q)[} (u_y)\right| \nu(du).
\end{align*}

Applying (\ref{int}), the last expression above is equal to
\begin{multline*}
	2d  \sum_{y\neq x} \int_U \left|\mathbbm{1}_{]0, \lambda(t,x,y,a,p)[} (u_y) - \mathbbm{1}_{]0, \lambda(s,x,y,b,q)[} (u_y)\right|
du_y \\
	= 2d  \sum_{y\neq x} | \lambda(t,x,y,a,p) -\lambda(s,x,y,b,q)|,
\end{multline*}
which gives the claims.
\end{proof}

In order to verify assumption (C), we need additional hypotheses on the structure of the model.

\begin{lemma}
Let $f$ be defined by (\ref{fcont}) and $\nu$  by (\ref{nu}). Assume that $A$ is a compact and convex subset of a metric topological vector space. Let the running cost $c$ be strictly convex in $a$ and the rate $\lambda$ appearing in (\ref{fcont}) be affine (in the sense of being both convex and concave) in $a$. Then assumption (C) is satisfied.
\label{ffeed}
\end{lemma} 

\begin{proof}
We have $H(t,x,a,p,g) = \Lambda^{a,p}_t g(x) + c(t,x,a,p)$ where
\[
	\Lambda^{a,p}_t g(x) = \int_U \left[g\left(x + \sum_{y\in \Sigma} (y-x) \mathbbm{1}_{]0, \lambda(t,x,y,a,p)[} (u_y)\right) -g(x)\right]\nu(du).
\]
Applying formula (\ref{intphi}) we obtain
\begin{align*}
	\Lambda^{a,p}_t g(x) &= \sum_{y=1}^d \int_{U_y} \left[g(x +  (y-x) \mathbbm{1}_{]0, \lambda(t,x,y,a,p)[} (u_y)) -g(x)\right] du_y \\
	&= \sum_{y=1}^d \lambda(t,x,y,a,p) \left[g(y)- g(x)\right],
\end{align*}
which is an affine function of $a$ if $\lambda$ is affine in $a$. Therefore, $H$ is a strictly convex function of $a$ if $c$ is strictly convex, and thus it has a unique minimum in $A$.
\end{proof}

\section{Relaxed Mean Field Game Solutions}
\label{SectExistence}

\subsection{The space $\mathcal{L}$}

In order to prove the existence of solutions we use a fixed point theorem. First of all, we want to find a suitable space where all the flows of probability measures lie. Set $K:=2\nu(U)\sqrt{d}$ and denote by
\begin{equation}
\mathcal{L} := \left\{m:[0,T]\longrightarrow S: |m(t)-m(s)| \leq K|t-s| , \quad m(0) = m_0 \right\}
\label{L}
\end{equation}
the space of Lipschitz continuous flows of probability measures, with the same Lipschitz constant $K$
and initial point $m_0$. 
This space is easily seen to be convex and compact with respect to the uniform norm, thanks to the Ascoli-Arzel{\`a} theorem. The following lemma allows to restrict attention to flows of probability measures in $\mathcal{L}$. 

\begin{lemma}
Let $\alpha\in\mathcal{A}$, or $\rho\in\mathcal{R}$, and let $m:[0,T]\longrightarrow S$ be any deterministic flow of probability measures. Then the flow of the solution process $Flow(X_{\alpha,m})$, or $Flow(X_{\rho,m})$, is in $\mathcal{L}$.
\label{Llem}
\end{lemma}

\begin{proof}
We prove the claim for relaxed controls, so the conclusion follows also when considering the subset of ordinary controls. Let $g:\Sigma\longrightarrow\mathbb{R}$ be a function, which is then Lipschitz and bounded and can be viewed as 
a vector in $\mathbb{R}^d$. Denote $|g|_{\infty} := \max_{x\in\Sigma} g(x)$. Let $\rho\in\mathcal{R}$ and $m$ be fixed, and set 
$X=X_{\rho,m}$. The function $m:[0,T]\rightarrow S$ has a priori no regularity, except for being measurable.
By the Dynkin formula (\ref{dynrel}) we have, for any $0\leq s\leq t\leq T$,
\begin{align*}
	&E[g(X(t))] - E[g(X(s))]\\
	&= \int_s^t \int_U \int_A E [g(X(r) + f(r,X(r), u, a, m(r))) -g(X(r))] \rho_r(da)\nu(du)dr.
\end{align*}
Hence
\begin{align*}
	&|E[g(X(t))] - E[g(X(s))]|\\
	&\leq \int_s^t \int_U \int_AE |g(X(r) + f(r,X(r), u, a, m(r))) -g(X(r))| \rho_r(da)\nu(du)dr\\
	&\leq  \int_s^t \int_U \int_A  2|g|_{\infty} \rho_r(da)\nu(du)dr = 2\nu(U) |g|_{\infty} (t-s)
\end{align*}
thanks to the fact that $\rho_r$ is a probability measure on $A$ for any $r$.
Clearly, $E[g(X(t))]= g \cdot Law(X(t))$. Thus, for any $t$ and $s$,
\begin{align*}
	|Law(X(t))- Law(X(s))| 
	&= \sqrt{\sum_{x=1}^d  \left[e_x \cdot (Law(X(t))- Law(X(s)))\right]^2}\\
	&\leq \sqrt{\sum_{x=1}^d  |t-s|^2 4 |e_x|_{\infty}^2 \nu(U)^2 } = 2 \nu(U)\sqrt{d} |t-s|,
\end{align*}
which gives the claim. 
\end{proof}

\subsection{Existence of relaxed mean field game solutions}

\subsubsection{Tightness and continuity for $m$ fixed}

Consider a sequence of random variables
\begin{equation}
(X^n, \rho^n, \N_{\rho^n})
\label{seq}
\end{equation}
where $\rho^n$ is a relaxed control, $\N_{\rho^n}$ is the related relaxed Poisson measure and $X^n = X_{\rho^n, m}$, $m\in\mathcal{L}$ is fixed.
The state space of these random variables is $D([0,T],\Sigma)\times \mathcal{D} \times \mathcal{M}$, where 
$\mathcal{M} =\mathcal{M}([0,T]\times U \times A)$ denotes the set of finite positive measures on $[0,T]\times U \times A$ endowed with the topology of weak convergence. 

The following is of fundamental importance, and is similar to Theorem 13.2.1 in \citet[p.\,363]{kushnerdupuis01}.

\begin{theorem}
Assume (A) and (B). Then
\begin{enumerate}
	\item any sequence of the form (\ref{seq}) is tight;
	\item the limit in distribution $(X,\rho, \widetilde{\N})$ of any converging subsequence is such that $\widetilde{\N}$ is the relaxed Poisson measure related to the relaxed control $\rho$ and $X = X_{\rho,m}$ in distribution;
	\item $J(\rho,m)$ is continuous in $\rho$.
\end{enumerate}
\label{tight}
\end{theorem}

\begin{proof}
(1) The sequence of relaxed controls is tight as $\mathcal{D}$ is compact. For any $\epsilon>0$, the set
\[
	K_\epsilon := \left\{\Theta\in\mathcal{M} : \Theta([0,T]\times U \times A)\leq \frac{T \nu(U)}{\epsilon}  \right\}
\]
is compact in $\mathcal{M}$, since $[0,T]\times U \times A$ is compact. From (\ref{relcon}) and the martingale property (\ref{mart}), it follows that $N_{\rho^n}(t,U,A) - t\nu(U)$ is a martingale for any $n$ and so $E[\N_{\rho^n}(T,U,A)]= T\nu(U)$. Therefore, by Chebychev's inequality,
\[
	 P\left(\N_{\rho^n} \notin K_\epsilon \right) 
	= P \left(\N_{\rho^n}(T,U,A)> \frac{T \nu(U)}{\epsilon} \right) \leq E[\N_{\rho^n}(T,U,A)] \cdot \frac{\epsilon}{T \nu(U)} = \epsilon
\]
for any $n$, saying that the sequence of relaxed Poisson measures is tight. The properties of the stochastic integral give
\[
	E\left[\left.|X^n(\tau+h)- X^n(\tau)|^2\right| \mathcal{F}_\tau\right] =O(h)
\]
for any $\mathbb{F}$-stopping time $\tau$, uniformly in $n$, which yields the tightness of the processes in $D([0,T],\Sigma)$ by Aldous's criterion \citep{aldous78}.
	
(2) By abuse of notations, denote by $(X^n, \rho^n, \N_{\rho^n})$ the subsequence which converges in distribution to $(X, \rho, \widetilde{\N})$. From the martingale property (\ref{mart}), it follows that $\widetilde{\N}(t,U_0,A_0)- \nu(U_0)\rho(t,A_0)$ is a martingale for any Borel sets $A_0\subset A$ and $U_0\subset U$, where the limiting measure is defined on the canonical space and the filtration is the canonical filtration (both defined in Appendix A). 
The limit random measure $\widetilde{\N}$ is integer valued \citep[Theorem 15.7.4 in][]{kallenberg86}, so the uniqueness property says that $\widetilde{\N} = \N_\rho$ in distribution.
The claim $X=X_{\rho,m}$ in distribution will be shown also in the proof of Theorem \ref{teo}, where $m$ is not fixed, using the controlled martingale problem, so we do not repeat the argument here.
	
(3) $\lim_{n\rightarrow \infty} J(\rho_n,m) = J(\rho,m)$ since $c$ and $\psi$ are bounded and continuous by assumption (B).
\end{proof}

By the chattering lemma, which we will present later as Lemma~\ref{chat}, we have
\[
	\min_{\rho\in\mathcal{R}}J(\rho, m) =\inf_{\alpha\in\mathcal{A}} J(\alpha,m).
\]
The minimum on the left hand side exists by the above Theorem~\ref{tight}. The infimum on the right hand side is actually a minimum, too; see Theorem~\ref{exi} below, where the existence of optimal feedback controls will be shown. However, there might exist more optima among relaxed open-loop controls than among ordinary feedback controls.

\subsubsection{Fixed point argument}

Let $2^\mathcal{L}$ be the set of subsets of $\mathcal{L}$ and define the point-to-set map 
$\Phi: \mathcal{L}\longrightarrow 2 ^\mathcal{L}$
 by 
\begin{equation}
	\Phi(m):= \left\{Flow(X_{\rho,m}): J(\rho,m) \leq J(\sigma,m) \quad \forall\sigma\in \mathcal{R}\right\},\quad m\in\mathcal{L}.
\label{phi}
\end{equation}
A flow  $m\in\mathcal{L}$ is called a \emph{fixed point} of this point-to-set map if $m\in \Phi(m)$. We need this map since the optimal control is not necessarily unique. 

By construction, $\Phi$ has a fixed point if and only if there exists a relaxed solution to the mean field game, in the sense of Definition \ref{solrel}. In order to prove the existence of a fixed point, we are going to apply Theorem~1 in \citet{fan52}, which requires the following definition.

\begin{definition}
Let $\mathcal{L}$ be a metric space. A map $\Phi: \mathcal{L}\longrightarrow 2 ^\mathcal{L}$  is said to have \emph{closed graph} if 
$m_n\in \mathcal{L}$, $y_n\in\mathcal{L}$, $y_n \in \Phi(m_n)$ for any $n\in \mathbb{N}$ and $m_n\rightarrow m$, $y_n\rightarrow y$ in 
$\mathcal{L}$ implies $y\in \Phi(m)$.
\label{cg}
\end{definition}

\begin{proposition}[Ky Fan]
Let $\mathcal{L}$ be a non empty, compact and convex subset of a locally convex metric topological vector space. 
Let $\Phi: \mathcal{L}\longrightarrow 2 ^\mathcal{L}$ have closed graph and assume that $\Phi(m)$ is non empty and convex for any $m\in\mathcal{L}$.
Then the set of fixed points of $\Phi$ is non empty and compact.
\label{fan}
\end{proposition}

By means of this proposition we are now able to state and prove the following main theorem concerning existence of relaxed solutions, while uniqueness is not guaranteed. 

\begin{theorem}
Under assumptions (A) and (B) there exists at least one relaxed solution of the mean field game (\ref{mfgr}).
\label{teo}
\end{theorem}

\begin{proof}
We want to show the existence of a fixed point for the map $\Phi: \mathcal{L}\longrightarrow 2 ^\mathcal{L}$ defined in (\ref{phi}), applying Proposition \ref{fan}. Recall that any element of $\Phi(m)$ is in $\mathcal{L}$ by Lemma~\ref{Llem}, and the set $\mathcal{L}$ defined in (\ref{L})  is a compact and convex subset of 
$\mathcal{C}([0,T],S)$ endowed with the uniform norm. By Theorem \ref{tight}, $\Phi(m)$ is non empty for any $m$. It remains to prove that $\Phi(m)$ is convex and $\Phi$ has closed graph.

\textbf{$\Phi(m)$ is convex}. 
Let $m$ be fixed and let $\rho_1,\rho_2 \in\mathcal{R}$ be such that $Flow(X_{\rho_1,m})$ and $Flow(X_{\rho_2,m})$ 
belong to $\Phi(m)$, i.e. $\rho_1$ and $\rho_2$ are optimal controls for $m$, and take $\theta\in [0,1]$. Let $\zeta$ be a Bernoulli random variable with parameter $\theta$, $\mathcal{F}_0$ measurable and independent of $\rho_1$ and $\rho_2$. Define $\rho_3\in\mathcal{R}$ by
\[
	\rho_3([0,t[ \times E):= \rho_1 ([0,t[ \times E) \mathbbm{1}_{\{\zeta=1\}} + \rho_2([0,t[ \times E)\mathbbm{1}_{\{\zeta=0\}}
\]
for any $E \in\mathcal{B}(A)$ and $t\in[0,T]$. We have 
\begin{align*}
	E[G(X_{\rho_3,m})] &= E\left[\left. G(X_{\rho_3,m})\right|\zeta=1 \right] P(\zeta=1) + E\left[\left. G(X_{\rho_3,m})\right| \zeta=0\right]P(\zeta=0) \\
	&= \theta E[G(X_{\rho_1,m})] + (1-\theta) E[G(X_{\rho_2,m})]
\end{align*}
for every $G\in \mathcal{C}_b (D([0,T],\Sigma),\mathbb{R})$. This implies that 
\begin{equation}
	Law(X_{\rho_3,m})=\theta Law(X_{\rho_1,m}) +(1-\theta)Law(X_{\rho_2,m})
\label{the}
\end{equation}
and then in particular
\begin{equation}
Flow (X_{\rho_3,m}) = \theta Flow (X_{\rho_1,m}) + (1-\theta) Flow(X_{\rho_2,m}).
\label{theta}
\end{equation}
Since $\rho_1$ and $\rho_2$ are optimal for $m$ we have, thanks to (\ref{the}),
\begin{align*}
	J(\rho_3,m) &= J(\rho_1,m) P(\zeta=1) + J(\rho_2,m) P (\zeta =0)\\
	&\leq \theta J(\sigma,m) + (1-\theta)J(\sigma,m) = J(\sigma,m)
\end{align*}
for any $\sigma \in\mathcal{R}$, which means that also $\rho_3$ is optimal for $m$ and hence (\ref{theta}) says that $\Phi(m)$ is convex.

\textbf{$\Phi$ has closed graph}.
Let $m_n, y_n,m,y\in \mathcal{L}$ be such that $m_n\rightarrow m$, $y_n \rightarrow y$ in $\L$ and $y_n \in \Phi(m_n)$ for every $n\in \mathbb{N}$. We have to prove that $y\in \Phi(m)$. Let $\rho_n\in\mathcal{R}$ be optimal for $m_n$ and such that $y_n = Flow (X_{\rho_n,m_n})$. 
Set $X_n:= X_{\rho_n,m_n}$ and let $\N_n:= \N_{\rho_{n}}$ be the relaxed Poisson measure related to $\rho_n$.

The tightness of the sequence $(X_n,\rho_n, \N_n)$ is proved as in Theorem \ref{tight}. Let $(X_{n_k},\rho_{n_k}, \N_{n_k})$ be a subsequence which converges in distribution to $(X,\rho, \widetilde{\N})$. We have $\widetilde{\N}=\N_\rho$ in distribution, 
i.e. it is the relaxed Poisson measure related to $\rho$. In order to prove that $X=X_{\rho,m}$ in distribution, we use the controlled martingale problem formulation stated in Lemma \ref{conmar}, and hence let us assume that the processes are defined in the canonical space.

Property (\ref{prmar}) holds for $X_{n_k}$, $\rho_{n_k}$ and $m_{n_k}$, any $k\in \mathbb{N}$. Let $M^{n_k}_g$ denote the process defined by
\begin{equation*}
\begin{split}
&M_g^{n_k}(t)=  g(X_{n_k}(t))-g(X_{n_k}(0)) \\
&- \int_0^t\int_U\int_A \left[ g(X_{n_k}(s) + f(s,X_{n_k}(s), u, a, m_{n_k}(s))) -g(X_{n_k}(s))\right] \nu(du)\rho^{n_k}_s(da)ds,
\end{split}
\end{equation*}
for any $g\in \mathbb{R}^d$.
Property (\ref{prmar}) and the convergence in distribution of the sequence $(X_{n_k},\rho_{n_k}, \N_{n_k})$ imply that 
\begin{align*}
0 &= \lim_{k\rightarrow \infty} E\left[h(X_{n_k}(t_i); i\leq j ) (M_g^{n_k}(t+s)-M_g^{n_k}(t))\right] \\
&=E\left[h(X(t_i); i\leq j ) (M_g(t+s)-M_g(t))\right] 
\end{align*}
thanks to continuity assumption (A), uniform convergence of $m_n$ and (\ref{conti}). Therefore we have proved that $X=X_{\rho,m}$ in distribution.
 
Thus we obtain 
\[
	\lim_{k\to\infty} Law(X_{n_k}) =Law (X_{\rho,m}),
\]
which implies the convergence 
\[
	\lim_{k\to\infty} \sup_{t\in [0,T]}|Law(X_k(t))- Law(X(t))|=0,
\]
that is, $Flow(X_{n_k}) \rightarrow Flow (X)$ uniformly. The convergence is then proved along a subsequence, but by hypothesis the limit
$Flow(X_{n})\rightarrow y$ exists in $\mathcal{L}$, hence $y=Flow (X) = Flow (X_{\rho,m})$.

It remains to prove that $\rho$ is optimal for $m$. 
Again the convergence in distribution of the sequence $(X_{n_k},\rho_{n_k}, \N_{n_k})$ implies that 
$\lim_k J(\rho_{n_k}, m_{n_k}) = J(\rho,m)$
thanks to continuity assumption (B), uniform convergence of $m_n$ and (\ref{conti}). 
Then from the optimality of $\rho_n$ for $m_n$, i.e. $J(\rho_{n_k}, m_{n_k})\leq J(\sigma, m_{n_k})$ for every $\sigma\in\mathcal{R}$, taking the limit as $k\rightarrow\infty$ we get $J(\rho,m)\leq J(\sigma, m)$ for every $\sigma\in\mathcal{R}$, which means that $\rho$ is optimal for $m$ and thus $y= Flow (X_{\rho,m}) \in \Phi(m)$ as required.
\end{proof}

\subsection{Relaxed feedback mean field game solutions}

Theorem \ref{teo} provides a relaxed (open-loop) solution of the mean field game (\ref{mfgr}). Under the same assumptions we obtain here a relaxed feedback mean field game solution which has the same cost and flow of the open-loop one. This result is similar to 
Theorem 3.7 in \cite{lacker15} and will provide approximate feedback Nash equilibria for the $N$-player game.

\begin{theorem}
Assume (A) and (B) and let $\left(((\Omega, \mathcal{F}, P; \mathbb{F}), \rho, \xi, \N),m,X_{\rho,m}\right)$ be a relaxed mean field game solution. 
Then there exists a relaxed feedback control $\widehat{\gamma}\in \widehat{\mathbb{A}}$ such that the tuple
 $\left(((\Omega, \mathcal{F}, P; \mathbb{F}), \widehat{\gamma}, \xi, \N), m, X_{\widehat{\gamma},m}\right)$
is a  relaxed feedback mean field game solution; namely
\begin{eqnarray}
	Flow (X_{\widehat{\gamma},m}) &=&  Flow(X_{\rho,m})  = m,
	\label{fluxm}\\
	 J(\widehat{\gamma},m) &=&  J(\rho,m).
\end{eqnarray}
\label{teorf}
\end{theorem}

\begin{proof}
The flow $m\in\mathcal{L}$ is fixed and set $X=X_{\rho,m}$. We claim that there exists a measurable function $\widehat{\gamma}:[0,T]\times\Sigma\longrightarrow \mathcal{P}(A)$ such that
\[
	\widehat{\gamma}(t,X(t)) = E[\rho_t| X(t)] \qquad \ell \otimes P\text{-almost every } (t,\omega)\in [0,T]\times\Omega.
\]
This holds if and only if 
\begin{equation}
	\int_A \varphi(t,X(t),a) [\widehat{\gamma}(t,X(t))](da) = E\left[ \int_A \varphi(t,X(t),a)\rho_t(da) \big| X(t)\right] 
\label{phii}
\end{equation}
for any bounded and measurable $\varphi: [0,T]\times \Sigma \times A \longrightarrow \mathbb{R}$.
In order to construct $\widehat{\gamma}$, define the probability measure $\Theta$ on $[0,T]\times\Sigma\times A$ by
\[
	\Theta(C):= \frac1T E\left[\int_0^T\int_A \mathbbm{1}_C (t,X(t),a) \rho_t(da)dt\right],\quad C\in\mathcal{B}([0,T]\times\Sigma\times A).
\]
Then build $\widehat{\gamma}$ by disintegration of $\Theta$:
\[
	\Theta(dt,ds,da)= \Theta_1(dt,dx) [\widehat{\gamma}(t,x)](da)
\]
where $\Theta_1$ denotes the $[0,T]\times \Sigma$ marginal of $\Theta$ and $\widehat{\gamma}:[0,T]\times \Sigma \longrightarrow \mathcal{P}(A)$ is measurable.
Following  \cite{lacker15}, we show that such $\widehat{\gamma}$ satisfies (\ref{phii}): 
for every bounded and measurable $h:[0,T]\times \Sigma \longrightarrow \mathbb{R}$ we get
\begin{align*}
	&E\left[\int_0^T h(t,X(t)) \int_A \varphi(t,X(t),a) [\widehat{\gamma}(t,X(t))](da) dt \right]\\
	&= T \int_{[0,T]\times \Sigma}h(t,x) \int_A \varphi(t,x,a) [\widehat{\gamma}(t,x)](da) \Theta_1(dt,dx)\\
	&=T \int_{[0,T]\times \Sigma\times A}h(t,x) \varphi(t,x,a) \Theta(dt,dx,da)\\
	&= E\left[\int_0^T h(t,X(t)) \int_A \varphi(t,X(t),a) \rho_t(da) dt\right],
\end{align*}
which provides (\ref{phii}) thanks to Lemma 5.2 in \cite{brunickshreve13}.

Having $\widehat{\gamma}$, (\ref{phii}) yields
\begin{align*}
	\int_U\int_A &f(t,X(t),u,a,m(t))[\widehat{\gamma}(t,X(t))](da)\nu(du)\\
	&= E\left[\left.\int_U\int_A f(t,X(t),u,a,m(t))\rho_t(da)\nu(du)\right|X(t)\right]
\end{align*}
$\ell \otimes P$-almost everywhere. 

Then we solve equation (\ref{mfgrf}) in the same probability space of $X$, under the relaxed feedback control $\widehat{\gamma}$,
and denote by $Y= X_{\widehat{\gamma},m}$ its solution.
By the Dynkin formula (\ref{dynrel}), we have for any $g\in\mathbb{R}^d$,
\begin{align*}
	E\left[ g(X(t))\right] &= E\left[ g(\xi) \right] +E\left[ \int_0^t\int_A \Lambda_s^a g(X(s)) \rho_s(da) ds\right] \\
	&= E\left[ g(\xi) \right] +E\left[ E\left[\left. \int_0^t\int_A \Lambda_s^a g(X(s)) \rho_s(da)\right| X(s)\right]ds \right]
\end{align*}
and then thanks to (\ref{phii}) 
\begin{equation}
	E\left[ g(X(t))\right] = E\left[g(\xi)\right] +E\left[ \int_0^t\int_A \Lambda_s^a g(X(s)) [\widehat{\gamma}(s,X(s))](da) ds \right], 
\label{flowX}
\end{equation}
while Dynkin's formula for $Y$ yields
\begin{equation}
	E\left[ g(Y(t))\right] = E\left[ g(\xi)\right] +E\left[ \int_0^t\int_A \Lambda_s^a g(Y(s)) [\widehat{\gamma}(s,Y(s))](da) ds \right].
\label{flowY} 
\end{equation}
Comparing (\ref{flowX}) and (\ref{flowY}) we obtain that
$Law (X(t)) $ and $Law (Y(t))$, which are vectors in $S\subset \mathbb{R}^d$, satisfy the same \mbox{ODE} in integral form, namely
\[
	g \cdot \pi(t) = g \cdot Law(\xi) + \int_0^t \int_\Sigma \int_A \Lambda_s^a g(x) [\widehat{\gamma}(s,x)](da) [\pi(t)] (dx) ds,\quad t\in [0,T],
\]
for any $g\in\mathbb{R}^{d}$, the unknown being denoted by $\pi:[0,T]\rightarrow S$. Taking $g = e_j$, $j=1,\ldots, d$, the corresponding system of \mbox{ODEs}, which is clearly linear in $\pi$, has a unique absolutely continuous solution $\pi\in \mathcal{L}$, hence (\ref{fluxm}) is proved.

Similarly, (\ref{phii}) gives 
\begin{align*}
J(\rho,m) &= E\left[\int_0^T\int_A c(t,X(t),a,m(t))\rho_t(da) dt + \psi(X(T),m(T))\right] \\
&=E\left[\int_0^T\int_A c(t,X(t),a,m(t))[\widehat{\gamma}(s,X(s))](da) dt + \psi(X(T),m(T))\right] 
\end{align*}
and then we use (\ref{fluxm})  to conclude that
\begin{align*}
	J(\rho,m) &= E\left[\int_0^T\int_A c(t,Y(t),a,m(t))[\widehat{\gamma}(s,Y(s))](da) dt + \psi(Y(T),m(T))\right] \\
	&= J(\widehat{\gamma},m).
\end{align*}
\end{proof}

\section{Feedback Mean field Game Solutions}

\subsection{Feedback optimal control for $m$ fixed}

 We show the existence of an optimal non-relaxed feedback control $\gamma_m$ for $J(\alpha,m)$ for any $m$, using the verification theorem for the related \emph{Hamilton-Jacobi-Bellman} equation. 
Let $m\in \mathcal{L}$ be fixed.

For any $t\in[0,T]$, $x\in\Sigma$ and $\alpha\in\mathcal{A}$ let $X_\alpha^{t,x}$ be the solution to 
\begin{equation}
	X_\alpha^{t,x}(s) = x +\int_t^s \int_U f(r^-,X_\alpha^{t,x}(r^-), u, \alpha(r), m(r)) \N(dr,du)
\label{Xtx}
\end{equation}
and set
\[
	J(t,x,\alpha, m) := E\left[ \int_t^T c(s,X_\alpha^{t,x} (s), \alpha(s), m(s))ds + \psi(X_\alpha^{t,x}(T), m(T)) \right].
\]
Next, define the \emph{value function} by
\begin{equation}
	V_m(t,x) := \inf_{\alpha \in \mathcal{A}} J(t,x,\alpha, m).
\label{V}
\end{equation}
Recall that the generator was defined in (\ref{gen}) by
\[
	\Lambda^{a,p}_t g(x):= \int_U [g(x + f(t,x, u, a, p)) -g(x)] \nu(du)
\]
for any $t,x,a,p$ and $g\in\mathbb{R}^d$.
For a function $v=v(t,x)$ the generator will be applied to the space variable, i.e. denote 
$\Lambda^{a,p}_t v(t,x)= \Lambda^{a,p}_tv(t,\cdot)(x)$.

Thanks to Theorem~D.5 in \cite{hernandezlermalasserre96} on measurable selectors, there exists a feedback control $\gamma_m \in\mathbb{A}$ 
(i.e.\ measurable) such that
\begin{equation}
	\gamma_m (t,x) \in \argmin_{a\in A} \left\{ \Lambda^{a,m(t)}_t V_{m}(t,x)+ c(t,x,a,m(t))\right\},
\label{argmin}
\end{equation}
where $V_m$ is the value function (\ref{V}).
Let us remark that the above minimum exists for any $t$ and $x$ if (A) and (B) hold, as the right hand side turns out to be a continuous function of the variable $a$, since the value function is trivially Lipschitz continuous in $x$.

\begin{theorem}
Assume (A) and (B). Let $m\in \mathcal{L}$. Then any feedback control $\gamma_m$ defined by (\ref{argmin}) is optimal, that is, $J(\gamma_m,m)\leq J(\alpha,m)$ for any $\alpha\in\mathcal{A}$.
\label{exi}
\end{theorem}


In order to prove Theorem~\ref{exi}, we use the \emph{Hamilton-Jacobi-Bellman} equation of the problem (see, for instance, Chapter~3 in \citet{flemingsoner06}):
\begin{equation}
\begin{cases}
	\frac{\partial v}{\partial t}(t,x) + \inf_{a\in A}\left\{\Lambda^{a,m(t)}_t v(t,x) + c(t,x,a,m(t))\right\} =0 &\text{in } [0,T[\times \Sigma \\
	v(T, x) = \psi(x,m(T)) &\text{in } \Sigma
\end{cases}
\label{HJB1}
\end{equation}
for a function $v\!: [0,T]\times \Sigma \rightarrow \mathbb{R}$.
Let us define, for $g\in \mathbb{R}^{d} \equiv \{\Sigma \rightarrow \mathbb{R}\}$,
\[
	G(t,x,g):= \inf_{a\in A} \left\{\int_U [g(x + f(t,x, u, a, m(t))) -g(x)] \nu(du) + c(t,x,a,m(t))\right\}.
\]
Since $\Sigma$ is finite, we shall denote $W_x(t):= v(t,x)$, $W(t):= (W_1(t),\ldots, W_d(t))$, $F_x(t,g):=G(t,x,g)$, $F(t,g):= (F_1(t,g),\ldots, F_d(t,g))$, $\Psi_x:= \psi(x,m(T))$ and $\Psi:= (\Psi_1(t),\ldots, \Psi_d(t))$. 
Therefore (\ref{HJB1}) can be written as
\begin{equation}
\begin{cases}
	\frac{d}{d t}W(t) + F(t,W(t)) =0, & t\in [0,T[,\\
	W(T) = \Psi, & 
	\end{cases}
\label{HJB4}
\end{equation}
which is in fact an \mbox{ODE}.

Define a \emph{classical solution} to (\ref{HJB4}) as an absolutely continuous function   $W$ from $[0,T]$ to $\mathbb{R}^d$ such that
$W(t)= \Psi + \int_t^T F(s,W(s))ds$ for every $t\in [0,T]$.
We apply to our problem the following verification theorem, which is a version of Theorem~3.8.1 in \citet[p.\,135]{flemingsoner06}:

\begin{proposition}[Verification]
Let $v$ be a classical solution to (\ref{HJB4}), and let $\gamma_m$ be any feedback control such that (\ref{argmin}) holds for Lebesgue almost every $t$.
Then
\[
	v(t,x)= J(t,x,\gamma_m,m) = V_m(t,x)
\]
for any $t\in[0,T]$ and $x\in \Sigma$, where $V_m$ is the value function (\ref{V}).
\label{veri}
\end{proposition}


We are now in the position to prove Theorem \ref{exi}. 

\begin{proof}[Proof of Theorem \ref{exi}] In view of Proposition \ref{veri}, we have just to show that there exists a classical solution to (\ref{HJB4}). Hence it is enough to prove that $F=F(t,w)$ is globally Lipschitz continuous in $w\in\mathbb{R}^d$, uniformly in $t\in[0,T]$. So let $t$ be fixed and take $w,z \in \mathbb{R}^d$ and $x\in\Sigma$. Recall that
\[
	F_x(t,w):= \min_{a\in A} \left\{\int_U [w_{x + f(t,x, u, a, m(t)))} -w_x] \nu(du) + c(t,x,a,m(t))\right\},
\]
and let $b$ be a minimizer for $F_x(t,z)$. Then
\begin{align*}
F_x(t,w)&- F_x(t,z)= \min_{a\in A} \left\{\int_U [w_{x + f(t,x, u, a, m(t)))} -w_x] \nu(du) + c(t,x,a,m(t))\right\} \\
&-\int_U [z_{x + f(t,x, u, b, m(t)))} -z_x] \nu(du) - c(t,x,b,m(t))\\
&\leq  \int_U [w_{x + f(t,x, u, b, m(t)))} -w_x] \nu(du) -\int_U [z_{x + f(t,x, u, b, m(t)))} -z_x] \nu(du)\\
&\leq \int_U \left|w_{x + f(t,x, u, b, m(t)))} -w_x - z_{x + f(t,x, u, b, m(t)))} +z_x\right|\nu(du)\\
&= \int_U \left|(w-z)_{x + f(t,x, u, b, m(t)))} -(w-z)_x \right|\nu(du) \leq 2 \nu(U) \max_{y\in\Sigma} \left|(w-z)_y \right|.
\end{align*}
Changing the role of $w$ and $z$ we obtain the converse, hence
\[
	\left|F_x(t,w)- F_x(t,z)\right|\leq 2 \nu(U) \max_{y\in\Sigma} \left|(w-z)_y \right|
\]
for any $x$, which implies
\[
	\max_{x\in\Sigma}\left|F_x(t,w)- F_x(t,z)\right|\leq 2 \nu(U) \max_{y\in\Sigma} \left|(w-z)_y \right|.
\]
Therefore $F$ is Lipschitz continuous in $w$ in the norm $||w||=\max_{y\in\Sigma} \left|w_y \right|$, which is equivalent to the Euclidean norm in $\mathbb{R}^d$.
\end{proof}

\subsection{Uniqueness of the feedback control for $m$ fixed}

Consider the \emph{pre-Hamiltonian}, as defined in (\ref{H}),
\[
	H(t,x,a,p,g) := \int_U \left[g(x + f(t,x, u, a, p)) -g(x)\right]\nu(du) + c(t,x,a,p)
\]
for $(t,x,a,p)\in [0,T]\times \Sigma\times A\times S$ and $g\in\mathbb{R}^d$. We make the additional assumption (C); so let us recall that $a^\ast(t,x,p,g)$ is the unique minimizer of $H(t,x,a,p,g)$ in $a\in A$. Define for $m\in\mathcal{L}$ the feedback control
\begin{equation}
	\gamma_m(t,x) := a^\ast(t,x,m(t),V_m(t,\cdot))
\label{gammam}
\end{equation}
where $V_m$ is the value function (\ref{V}). 

\begin{theorem}
	Assume (A), (B) and (C).
	Given $m\in \mathcal{L}$, let $\sigma\in \R$ be any optimal relaxed control for $m$ and let $X_{\sigma,m}$ be the corresponding solution to (\ref{mfgr}). 
	Then $\sigma_t = \gamma_m(t, X_{\sigma,m}(t))$ for $\ell\otimes P$-almost every $(t,\omega)$, that is, $\sigma$ corresponds to the feedback control $\gamma_m$.
\label{teo2}
\end{theorem}

This result and the proof of Theorem \ref{teo} imply that any relaxed solution of the mean field game must correspond to a feedback solution:

\begin{corollary}
	Assume (A), (B) and (C).
	Then there exists a feedback solution $(\gamma, m, X)$ of the mean field game, and any solution is such that its control coincides with $\gamma_m$.
\label{coro}
\end{corollary}

Let $Q\in \mathcal{P}(A)$, and define
\[
	\widetilde{H}(t,x,Q,p,g):= \int_A H(t,x,a,p,g) Q(da).
\]

\begin{lemma}
If $H$ is continuous in $a$, then 
\begin{equation}
	\min_{Q\in \mathcal{P}(A)} \widetilde{H}(t,x,Q,p,g) = \min_{a\in A} H(t,x,a,p,g)
\label{Htilde}
\end{equation}
for any $t,x,p$ and $g$.
Moreover, if (C) holds, then there exists a unique $Q^\ast \in\mathcal{P}(A)$ such that
\[
	\widetilde{H}(t,x,Q^\ast,p,g) = \min_{Q\in \mathcal{P}(A)} \widetilde{H}(t,x,Q,p,g) = \min_{a\in A} H(t,x,a,p,g) = H(t,x,a^\ast,p,g)
\]
and $Q^\ast = \delta_{a^\ast}$, where $a^\ast = a^\ast(t,x,p,g)$.
\label{Q}
\end{lemma}

\begin{proof}
If $H$ is continuous in $a$, then $\widetilde{H}$ is continuous in $Q\in\mathcal{P}(A)$ in the weak topology. Since $\mathcal{P}(A)$ is compact, there exists a minimum: let $Q^\ast$ be a minimizer. 
For fixed $t,x,p$ and $g$ we have 
\[
	\min_{Q\in \mathcal{P}(A)} \widetilde{H}(t,x,Q,p,g) \leq \min_{Q= \delta_a , a\in A } \int_A H(t,x,a,p,g) Q(da) = \min_{a\in A} H(t,x,a,p,g)
\]
and
\[
	\min_{Q\in \mathcal{P}(A)} \int_A H(t,x,a,p,g) Q(da) \geq \min_{Q\in \mathcal{P}(A)} \int_A H(t,x,a^\ast,p,g) Q(da) = H(t,x,a^\ast,p,g),
\]
which means that $\widetilde{H}(t,x,Q^\ast,p,g) =H(t,x,a^\ast,p,g)$.

Consider $H(t,x,a,p,g) - H(t,x,a^\ast,p,g)$ as a function of $a$: it is non-negative and, if (C) holds, it equals zero if and only if $a= a^\ast$. Therefore,
\[
	0= \tilde{H}(t,x,Q^\ast,p,g) -H(t,x,a^\ast,p,g) = \int_A \left[H(t,x,a,p,g) - H(t,x,a^\ast,p,g)\right] Q^\ast(da),
\]
which implies the claim, namely that $Q^\ast (\left\{a^\ast\right\}) =1$.
\end{proof}

\begin{remark}
Note that if (C) does not hold, then $Q^\ast$ is supported on the set of all minimizers of $H$. Thus it might not be a Dirac measure.
This  implies that there may exist an optimal relaxed control which is not an ordinary control (not even open-loop).
\end{remark}

%

\begin{proof}[Proof of Theorem~\ref{teo2}]

Fix $m\in \mathcal{L}$. Let $\sigma\in\mathcal{R}$ be an optimal relaxed control and denote by $X_\sigma = X_{\sigma,m}$ the corresponding optimal trajectory. By the chattering lemma, which we will state later as Lemma~\ref{chat}\footnote{Here only the open loop part of the chattering lemma is needed, which is well known, and so we postpone the proof of the lemma to Section~5, where we also give the feedback part.},
\begin{multline*}
	E\left[ V(0,X_{\sigma}(0))\right] = \min_{\rho\in \mathcal{R}} J(\rho,m) = J(\sigma,m) \\
	= E\left[\int_0^T\int_A c(t,X_{\sigma}(t),a,m(t))\sigma_t(da) dt + \psi(X_{\sigma}(T),m(T))\right],
\end{multline*}
where $V = V_{m}$ is the value function defined in (\ref{V}). Thanks to (\ref{HJB1}), the Hamilton-Jacobi-Bellman equation, and (\ref{Htilde}), we have 
\begin{equation} \label{VHtildeIneq}
	\frac{\partial}{\partial t}V(t,x) + \tilde{H}(t,x,\sigma_t, V(t,\cdot)) \geq 0 \quad\text{for all } t,x,\omega.
\end{equation}
By the Dynkin formula (\ref{dynrel}) and the terminal condition for $V$,
\begin{align*}
\begin{split}
	E\left[ V(0,X_{\sigma}(0))\right] &= E\left[V(T,X_{\sigma}(T))\right] \\
	&\quad - E\left[\int_0^T \left( \frac{\partial}{\partial t}V(t,X_{\sigma}(t)) + \int_A \Lambda_{t}^{a,m(t)}V(t,X_{\sigma}(t)) \sigma_t(da)\right) dt \right]
\end{split}\\
\begin{split}
	&= E\left[\int_0^T\int_A c(t,X_{\sigma}(t),a,m(t))\sigma_t(da) dt + \psi(X_{\sigma}(T),m(T))\right] \\
	&\quad - E\left[ \int_{0}^{T} \frac{\partial}{\partial t}V(t,X_{\sigma}(t)) + \tilde{H}(t,X_{\sigma}(t),\sigma_t, V(t,\cdot))dt \right].
\end{split}
\end{align*}
It follows that
\[
	E\left[ \int_{0}^{T} \frac{\partial}{\partial t}V(t,X_{\sigma}(t)) + \tilde{H}(t,X_{\sigma}(t),\sigma_t, V(t,\cdot))dt \right] = 0,
\]
hence, in view of (\ref{VHtildeIneq}),
\[
	\frac{\partial}{\partial t}V(t,X_{\sigma}(t)) + \tilde{H}(t,X_{\sigma}(t),\sigma_t, V(t,\cdot)) = 0
\]
for $\ell\otimes P$-almost every $(t,\omega)$, which means that
\[
	\sigma_t  \in \argmin_{Q\in \mathcal{P}(A)} \tilde{H}(t,X_{\sigma,m}(t),\sigma_t, V(t,\cdot))
\]
for $\ell\otimes P$-almost every $(t,\omega)$. If (C) holds, then, by Lemma~\ref{Q}, the unique minimizer of $Q\mapsto \widetilde{H}(t,x,Q,m(t),V(t,\cdot))$ is the measure $Q^\ast = \delta_{a^\ast} \in \mathcal{P}(A)$ with 
$a^\ast = a^\ast(t,x,m(t),V(t,\cdot))$. It follows that $\sigma_t = \gamma_m(t, X_{\sigma}(t))$ for $\ell\otimes P$-almost every $(t,\omega)$.

\end{proof}


\subsection{Uniqueness of the feedback MFG solution for small time}

In this subsection, we focus only on the dynamics for $f$ in (\ref{fcont}),
\[
	f(t,x,u,a,p) := \sum_{y\in \Sigma} (y-x) \mathbbm{1}_{]0, \lambda(t,x,y,a,p)[} (u_y),
\]
and $\nu$ defined in (\ref{nu}) with $U:= [0,M]^d$. Moreover, we assume that $A = U$ and, for $x\neq y$,
\[
	\lambda(t,x,y,a,p) = a_y  + \zeta(p),
\]
where $\zeta:S \rightarrow \mathbb{R}$ is some Lipschitz continuous function with Lipschitz constant $K_{\zeta}$ such that $\zeta(p)\geq \kappa$ for some $\kappa>0$. Since $\lambda$ determines the transition rates, we set
$\lambda(t,x,x,a,p) := - \sum_{y\neq x} \lambda(t,x,y,a,p)$, $x\in \Sigma$.
 
We assume that the cost $c$ in the variable $a$ is in $\mathcal{C}^1(A)$, $\nabla_a c$ is Lipschitz continuous in the variable $p$ with Lipschitz constant $K_a$
and $c$ is uniformly convex, that is, there exists $\theta>0$ such that
\[
	c(t,x,b,p) - c(t,x,a,p) \geq \nabla_a c(t,x,a,p) \cdot (b-a) + \theta|b-a|^2
\]
for all $t,x,a,b,p$.

This setup is analogous to the one considered in \cite{gomesetal13}. The assumptions of Lemma~\ref{ffeed} are satisfied and thus for any $g\in \mathbb{R}^d$ there exist a unique minimizer 
$a^\ast(t,x,p,g)$ of $H(t,x,a,p,g)$, which in this setting becomes
\begin{align}
	H(t,x,a,p,g) &= \sum_{y=1}^d \lambda(t,x,y,a,p) [g(y)- g(x)] + c(t,x,a,p)\label{HH} \\
&= \sum_{y=1}^d (a_y + \zeta(p))[g(y)- g(x)]  + c(t,x,a,p).\nonumber
\end{align}
The assumptions of Lemma~1 are satisfied so that (A'') and (B'') hold.
We need $a^\ast$ to be Lipschitz continuous in $p$ and $g$; this fact is proved in Proposition~1 in \cite{gomesetal13}. We state the result in the following

\begin{lemma}
Under the above assumptions (in this subsection), the function $a^\ast$ is Lipschitz continuous in $p$ and $g$:
\begin{align}
	| a^\ast(t,x,p,g) - a^\ast(t,x,q,g)| &\leq \frac{K_a}{\theta} |p-q|
\label{ap}\\
	| a^\ast(t,x,p,g) - a^\ast(t,x,p,h)| &\leq \frac{1}{\theta} |g-h|
\label{ag}
\end{align}
for any $t,x,p,q,g,h$.
\label{alip}
\end{lemma}

Let us fix here the filtered probability space, the initial condition and the Poisson random measure.
Define $\gamma_m(t,x) = a^\ast(t,x,m(t),V_m(t,\cdot))$ as in (\ref{gammam}): 
it is the unique feedback control for given flow of measures $m\in \mathcal{L}$,
where $V_m(t,x)$ is  the value function defined in (\ref{V}) with respect to $m$. 
The cost functions $c$ and $\psi$ are uniformly bounded and so is the value function: Let us denote by $M_V$ the maximum of its absolute value. Denote by $M_\zeta$ the maximum of $\zeta$ and fix the constants
\begin{align*}
	C_1&:= 2 Md^2 +2d\sqrt{d} M^d,\\
	C_2&:= 2d\sqrt{d}\frac{K_a}{\theta} + 2d^2 K_\zeta,\\
	C_3&:= \frac{2d^2}{\theta},\\
	C_4&:= K_2 + 2d M_V K_\zeta + 2\sqrt{d} M_V \frac{K_a}{\theta} + K_2 \frac{K_a}{\theta},\\
	C_5&:= 2 M_V \frac{\sqrt{d}}{\theta} +\frac{K_2}{\theta} + \sqrt{d} (M_\zeta +M).
\end{align*}

Let $T^\ast>0$ be such that
\begin{equation}
	2T^\ast \sqrt{d}  e^{T^\ast C_1} \left[C_2 + C_3 (K_2 +T^\ast C_4) e^{T^\ast C_5}\right] =1.
\label{Tstar}
\end{equation}

\begin{theorem}
Under the assumptions of this subsection, for any $0<T< T^\ast$ there exists a unique feedback solution 
$(\gamma, m, X)$ of the mean field game. It is such that $\gamma$ is the feedback control $\gamma_m$.
\label{uniq}
\end{theorem}

\begin{proof}
In the notation of Theorem \ref{teo}, the map $\Phi:\mathcal{L}\rightarrow \mathcal{L}$ is defined by
$\Phi(m) = \left\{Flow( X_{\gamma_m, m})\right\}$, a singleton.
If we prove that this map is a contraction for small time horizon $T$, then the assertion follows by the Banach-Cacciopoli Theorem. 
So let $m,n \in \mathcal{L}$ and set $X:=X_{\gamma_m, m}$ and $Y:=X_{\gamma_n, n}$. 
For a vector $v\in\mathbb{R}^d$ denote  $|v|_{\infty} = \max_{x\in\Sigma} |v_x|$.
 
First we prove that the value function $V_m$ is Lipschitz continuous with respect to $m$. Thanks to the \mbox{HJB} equation (\ref{HJB1}) we have
\[
	V_m(t,x)= V_m(T,x) +\int_t^T  H(s,x, a^\ast(s,x,m(s), V_m(s,\cdot)),m(s), V_m(s,\cdot))ds.
\]
The Hamiltonian $H$ is Lipschitz in $(a,p,g)$; in fact, by (\ref{lipD}) and (\ref{HH}) we have
\begin{align*}
&|H(t,x,a,p,g) - H(t,x,b,q,h)|\\
&\leq 2|g|_{\infty} \left(\sqrt{d} |a-b| +d K_\zeta |p-q| \right)
+ K_2(|a-b| +|p-q|) + \sqrt{d}( M_\zeta +M) |g-h|.
\end{align*}

Then using (\ref{ap}) and (\ref{ag}) we obtain
\begin{align*}
|V_m(t,x)& - V_n(t,x)|\\ 
&\leq K_2|m(T)-n(T)| + \int_t^T\left[ C_4|m(s)-n(s)| + C_5|V_m(s)-V_n(s)|\right] ds\\
&\leq K_2 |m(T)-n(T)| + C_4(T-t)||m-n||_{\infty} + \int_t^T C_5|V_m(s)-V_n(s)|ds 
\end{align*}
for any $x$, hence Gronwall's lemma implies that
\[
	|V_m(t) - V_n(t)|\leq \sqrt{d}(K_2+T C_4) e^{T C_5}||m-n||_{\infty}
\]
for any $0\leq t\leq T$.

Therefore, by applying again (\ref{ap}) and (\ref{ag}), we obtain 
\begin{align*}
&E|X(t) - Y(t)|\\
& \leq\int_0^t\int_U E|f(s, X(s), u, \gamma(s, X(s), m(s),V_m(t,\cdot)), m(s)) \\
&\qquad-f(s, Y(s), u, \gamma(s, Y(s), n(s),V_n(t,\cdot)), n(s))| \nu(du) ds\\
&\leq  \int_0^t\left[ C_1 E|X(s) - Y(s)| + C_2 |m(s)-n(s)|  +
\frac{2d\sqrt{d}}{\theta} |V_m(t,\cdot) - V_n(t,\cdot)| \right]ds\\
&\leq  \int_0^t \left[ C_1 E|X(s) - Y(s)| + C_2 |m(s)-n(s)| + C_3(K_2 + C_4 T)e^{T C_5} ||m-n||_{\infty} \right] ds
\end{align*}
and thus, again by Gronwall's lemma, 
\[
	E|X(t) - Y(t)| \leq T^\ast e^{T^\ast C_1} \left[C_2 + C_3 (K_2 +T^\ast C_4) e^{T^\ast C_5}\right]||m-n||_{\infty}
\]
for any $0\leq t\leq T$. Since $|Law (X(t)) - Law (Y(t))| \leq 2\sqrt{d} E|X(t)-Y(t)|$ we have
\begin{align*}
	\sup_{0\leq t\leq T} &|Law (X(t)) - Law (Y(t))| =: ||Flow(X) - Flow(Y)||_{\infty} \\
	&\leq 2 T^\ast\sqrt{d} e^{T^\ast C_1} \left[C_2 + C_3 (K_2 +T^\ast C_4) e^{T^\ast C_5}\right]||m-n||_{\infty},
\end{align*}
and then the claim holds for $T^\ast$ satisfying (\ref{Tstar}). 
\end{proof}

\subsection{Uniqueness under monotonicity}

Uniqueness of mean field game solutions was shown in Theorem 2 in \cite{gomesetal15} for arbitrary time horizon under the Lasry-Lions monotonicity assumptions. Here, we give a different proof of this result, which relies on the probabilistic representation of the mean field game, and allows for less restrictive assumptions on the data.

Specifically, we suppose that the function $f$ in the dynamics \eqref{mfg} does not depend on $p\in S$ and that the running cost $c$ splits in 
$c(t,x,a,p)= c_0(t,x,a)  + c_1(x,p)$. Moreover we assume that $c_1$ and $\psi$ satisfy the following monotonicity property:
\begin{align}
	&\sum_{x\in\Sigma}(c_1(x,p)-c_1(x,p'))(p_x-p'_x)> 0
	\label{mon1}\\
	&\sum_{x\in\Sigma}(\psi(x,p)-\psi(x,p'))(p_x-p'_x)\geq 0
	\label{mon2}
\end{align}
for any $p\neq p'\in S$. For example, $c_1$ and $\psi$ could be the gradient of convex functions in $\mathbb{R}^d$.

\begin{theorem} \label{ThMonotonicity}
Suppose that (A), (B) and the assumptions above hold. Let $(\gamma,m,X)$ and $(\gamma',m',X')$ be two feedback mean field game solutions. 
Then $m(t)=m'(t)$ for any $t$. Also the corresponding value functions $V_m$ and $V_{m'}$ are the same. 
Moreover, if (C) holds, then $\gamma(t,x)=\gamma'(t,x)$ for any $t,x$.
\end{theorem}

\begin{proof}
Since the dynamics does not depend on $p\in S$, we have $X=X^{\gamma}=X^{\gamma,m}=X^{\gamma,m'}$ and 
$X'=X^{\gamma'}=X^{\gamma',m}=X^{\gamma',m'}$. 
The optimality of $\gamma$ yields $J(\gamma,m)\leq J(\gamma',m)$ and similarly $J(\gamma',m')\leq J(\gamma,m')$, hence
\begin{align*}
0&\leq J(\gamma',m) -J(\gamma,m) = E\left[\psi(X'(T),m(T))-\psi(X(T),m(T))\right]\\
&+E\left[\int_0^T [c_0(X'(t),\gamma'(t,X'(t)))+c_1(X'(t),m(t))\right.\\
&\left. \qquad -c_0(X(t),\gamma(t,X(t)))-c_1(X(t),m(t))]dt\right]\\
0&\leq J(\gamma,m') -J(\gamma',m') =E\left[\psi(X(T),m'(T))-\psi(X'(T),m'(T))\right]\\
&+E\left[\int_0^T [c_0(X(t),\gamma(t,X(t)))+c_1(X(t),m'(t))\right.\\
&\left. \qquad -c_0(X'(t),\gamma'(t,X'(t)))-c_1(X'(t),m'(t))]dt\right].
\end{align*}
Summing these two inequalities and using the fact that $Law(X(t))=m(t)$ for any $t$, we obtain
\begin{align*}
0&\leq E\left[\psi(X'(T),m(T))-\psi(X(T),m(T)) + \psi(X(T),m'(T))-\psi(X'(T),m'(T))\right]\\
&+E\left[\int_0^T [c_1(X'(t),m(t)) -c_1(X(t),m(t)) \right.\\
&\left. \qquad +c_1(X(t),m'(t))-c_1(X'(t),m'(t))]dt\right]\\
&= \sum_{x\in\Sigma}(\psi(x,m(T))-\psi(x,m'(T)))(m'_x(T)-m_x(T))\\
&+\int_0^T \left[\sum_{x\in\Sigma}(c_1(x,m(t))-c_1(x,m'(t)))(m'_x(t)-m_x(t))\right]dt.
\end{align*}
If $m(t)\neq m'(t)$ for some $t$, then the latter expression is $<0$, thanks to \eqref{mon1}, \eqref{mon2} and the continuity of $m$; a contradiction. Therefore $m(t)=m'(t)$ for all $t$. 

The fact that $V_m = V_{m'}$ is implied by the uniqueness of solutions to the HJB equation \eqref{HJB1}.
Assuming (A) and (B), the optimal feedback $\gamma$ satisfies \eqref{argmin}. Thus, if (C) holds, then $\gamma=\gamma'$.
\end{proof}

\section{Approximation of $N$-player game}

\subsection{Approximation of relaxed controls}

In order to get an $\epsilon$-Nash equilibrium for the $N$-player game in open-loop strategies, respectively in feedback strategies, we have first to find an approximation of the optimal relaxed control, respectively relaxed feedback control, for the mean field game. To this end, we will make use of the following version of the chattering lemma.

\begin{lemma}[Chattering]

For any relaxed control $\rho\in\mathcal{R}$, there exists a sequence of stochastic open-loop controls $\alpha_n \in\mathcal{A}$ such that,
denoting by $\rho^{\alpha_n}(dt,da)=\delta_{\alpha_n(t)}(da) dt$ their relaxed control representation,
\[
	\lim_{n\rightarrow\infty} \rho^{\alpha_n} = \rho \quad P\text{-a.s.},
\]
where the limit is in the weak topology in $\mathcal{M}([0,T]\times A)$. 
Moreover, any $\alpha_n$ takes values in a finite subset of $A$.

For any relaxed feedback control $\widehat{\gamma}\in\widehat{\mathbb{A}}$, there exists a sequence of feedback controls 
$\gamma_n \in\mathbb{A}$ such that
\begin{equation}
	\lim_{n\rightarrow\infty} \delta_{\gamma_n(t,x)}(da) dt = [\widehat{\gamma}(t,x)](da)dt 
\label{convdel}
\end{equation}
uniformly in $x\in \Sigma$ and 
\begin{equation}
	\lim_{n\rightarrow\infty} \rho^{\gamma_n} = \rho^{\widehat{\gamma}} \quad \text{in distribution},
\label{convf}
\end{equation}
where $\rho^{\gamma_n}$ denotes the relaxed control representation of the open-loop control $\alpha^{\gamma_n}$ corresponding 
to $\gamma^n$, as in (\ref{algam}), and  $\rho^{\widehat{\gamma}}$ is defined in (\ref{rgg}); 
i.e. $\rho^{\gamma_n}_t (da) = \delta_{\gamma_n(t,X_{\gamma_n}(t^-)}(da)$ and 
$\rho^{\widehat{\gamma}}_t (da) = [\widehat{\gamma}(t, X_{\widehat{\gamma}}(t^-))](da)$.

\label{chat}
\end{lemma}

\begin{proof}
The first part is proved as Theorem 3.5.2 in \citet[p.\,59]{kushner90}, and the construction of the approximating sequence in the proof gives the $\gamma_n$ for the second part; let us show how to build them. 
Let $\widehat{\gamma}\in\widehat{\mathbb{A}}$, cover $A$ by $M_r$ disjoint sets $C_i^r$ which contain a point $a_i^r$ and set
$A^r := \left\{a_i^r : i\leq M_r \right\}$, a finite subset of $A$. For any $\Delta>0$ and $i,j$ define the function
\[
	\tau_{ij}^{\Delta r}(x) := \int_{i\Delta}^{(i+1)\Delta} [\widehat{\gamma}(s,x)](C_j^r) ds.
\]
Divide any interval $[(i+1)\Delta, (i+2)\Delta[$ into $M^r$ subintervals $I_{ij}^{\Delta r}(x)$ of length $\tau_{ij}^{\Delta r}(x)$ and define the feedback control $\gamma^{\Delta r}$, which is piecewise constant, by
\[
\gamma^{\Delta r}(t,x) := 
\begin{cases}
	a_0  \quad &t\in[0,\Delta[\\
	a_j^r \quad&t\in I_{ij}^{\Delta r}(x)  
\end{cases}
\]
where $a_0$ is an arbitrary value in $A$. The proof in \citet{kushner90} shows that 
\[
	\lim_{\substack{r\rightarrow0 \\ \Delta\rightarrow0}} \delta_{\gamma^{\Delta r}(t,x)}(da)dt = [\widehat{\gamma} (t,x)](da)dt
\]
weakly, for any $x\in\Sigma$. Since $\Sigma$ is finite we obtain that there exists a sequence of ordinary feedback controls $(\gamma_n)$ such that (\ref{convdel}) holds uniformly in $x$.
Let $m\in\mathcal{L}$ be fixed and $X_n$ be the solution to (\ref{Xfeed}) corresponding to the feedback control $\gamma_n$. 
By Theorem~\ref{tight}, the sequence $X_n$ is tight and there are a subsequence, which we still denote as $(X_{n})$, and a process $X$ such that 
$\lim_{n\rightarrow\infty}X_{n} =X$ in distribution. Possibly applying the Skorokhod representation \citep[Theorem 4.30 in][p.\,79]{K2001},
we may assume that this convergence is with probability one in the space of c{\`a}dl{\`a}g functions $D([0,T],\Sigma)$ equipped with the Skorokhod metric. This implies in particular that 
\begin{equation}
	P\left(\lim_{n\rightarrow\infty}X_{n}(t) = X(t) \mbox{ for any } t\notin\overline{\eta} \right)=1,
\label{almost}
\end{equation}
where $\overline{\eta}$ is the finite random set of discontinuity points (the jumps) of $X$.

Let now $\varphi\in\mathcal{C}([0,T]\times A)$ be any continuous function, which is also bounded as $A$ is compact.
We have to show the convergence to zero, almost surely, of
\[
	\int_0^T\int_A \varphi(t,a) [\delta_{\gamma_n(t,X_{n}(t^-))} - \widehat{\gamma} (t,X(t^-))](da)dt = Y_n +Z_n,
\]
where
\begin{align*}
Y_n &=\int_0^T \left[\varphi(t,\gamma_n(t,X_n(t^-)))-\varphi(t,\gamma_n(t,X(t^-)))\right]dt\\
Z_n &=\int_0^T\int_A \varphi(t,a)[\delta_{\gamma_n(t,X(t^-))}-\widehat{\gamma} (t,X(t^-))](da)dt.
\end{align*}
Any feedback control is Lipschitz in $x$, i.e. $dist(\gamma_n(t,x),\gamma_n(t,y)) \leq Diam(A) |x-y|$, and so $Y_n$ tends to zero thanks to 
(\ref{almost}), the continuity of $\varphi$ and dominated convergence.
As to $Z_n$, write 
$Z_n = \sum_{x\in\Sigma} Z_n^x$ where
\[
	Z_n^x := \int_0^T \int_A \mathbbm{1}_{\eta_x}(t) \varphi(t,a) [\delta_{\gamma_n(t,x)}  - \widehat{\gamma} (t,x)](da)dt
\]
and $\eta_x$ is the random set in $[0,T]$ where $X(t)=x$.
For each $x$, the random set $D_x$ of discontinuity points of the function $\mathbbm{1}_{\eta_x}(t) \varphi(t,a)$ is a subset of $\overline{\eta}_{x}\times A$ for some finite random set $\overline{\eta}_{x} \subset [0,T]$. Thus $D_x$ has null measure with respect to the limiting control $\widehat{\gamma} (t,x)(da)dt$ with probability one, for each $x$, thanks to Definition \ref{defrel}.
Hence by (\ref{convdel}) we get that $Z_n^x$ tends to zero for each $x$ and so does $Z_n$ since $\Sigma$ is finite.

Let $\alpha_{n}(t)= \gamma_n(t,X_{n}(t^-))$ be the open-loop control corresponding to $\gamma_{n}$ 
and $\rho^{n}$ its relaxed control representation.
We have just proved that $\lim_{n\rightarrow\infty} \rho^{n} = [\widehat{\gamma}(t,X(t^-))](da)dt$ $P$-almost surely and thus Theorem \ref{tight} says that $X$ must have the same law as the solution to (\ref{mfgrf}) under the relaxed feedback control $\widehat{\gamma}$. That solution is unique by Lemma~\ref{lemmaexi}, meaning that $X=X_{\widehat{\gamma}}$ in distribution. Therefore (\ref{convf}) follows since $\rho_t^{\widehat{\gamma}} = \widehat{\gamma}(t,X_{\widehat{\gamma}}(t^-))$ by (\ref{rgg}).  
\end{proof}

\begin{remark}
In the above proof we strongly used the finiteness of $\Sigma$  to get the approximation in feedback controls. While the result in the open-loop setting holds for general state space $\Sigma$, when considering feedback controls it is not clear whether the above lemma can be generalized to uncountably infinite state spaces.
\end{remark}

We are now able to state the approximation result:

\begin{proposition}
Let   $m\in\mathcal{L}$, $\rho\in\mathcal{R}$ and $\widehat{\gamma}\in\widehat{\mathbb{A}}$. Then for every $\epsilon>0$ there exist $\alpha\in\mathcal{A}$  and $\gamma \in\mathbb{A}$ such that  
\begin{align}
	E\left[\sup_{t\geq0}|X_{\alpha,m}(t) - X_{\rho,m}(t)| \right]&\leq \epsilon 
\label{meps}\\
	E\left[\sup_{t\geq0}|X_{\gamma,m}(t) - X_{\widehat{\gamma},m}(t)| \right]&\leq \epsilon 
\label{mepsf}\\
	|J(\alpha,m)-J(\rho,m)| &\leq \epsilon.
\label{Jeps}\\
	|J(\gamma,m)-J(\widehat{\gamma},m)| &\leq \epsilon.
\label{Jepsf}
\end{align}

\label{epsilon}
\end{proposition}

\begin{proof}
Let $(\alpha_n)$ be a sequence in $\mathcal{A}$ that approximates $\rho$ as in Lemma~\ref{chat}. Then we apply Theorem \ref{tight} to the sequence
$(X_{\alpha_n,m}, \alpha_n, m)$: it is tight, a subsequence $(X_{\alpha_{n_k},m}, \alpha_{n_k}, m)$ converges in distribution to $(X_{\rho,m},\rho,m)$ and $\lim_{k\rightarrow\infty} J(\alpha_{n_k}, m) = J(\rho,m)$. 
Thus there exist $\alpha_{n_k}=: \alpha$ for which (\ref{meps}) and (\ref{Jeps}) hold.
In a similar way, one proves (\ref{mepsf}) and (\ref{Jepsf}) for feedback controls.
\end{proof}

\subsection{$\epsilon_N$-Nash equilibria}

We can now define the approximate Nash equilibrium for the $N$-player game, first in open-loop form. 

\begin{notation}
Let $\left( \left((\Omega, \mathcal{F}, P; (\mathcal{F}_t)_{t\in [0,T]}), \alpha, \xi, \N \right), m, X_{\rho,m} \right)$ be a relaxed solution of the mean field game (\ref{mfgr}), which exists assuming (A) and (B) by Theorem \ref{teo}.
Fix $N\in\mathbb{N}$ and let $\alpha\in\mathcal{A}$ be as in Proposition \ref{epsilon},  satisfying  (\ref{meps}) and (\ref{Jeps}) with $\epsilon = \frac{1}{\sqrt{N}}$.
Then  $\left((\Omega, \mathcal{F}, P; (\mathcal{F}_t)_{t\in [0,T]}), \alpha^N, \xi^N, \N^N  \right)$ denotes the strategy vector where $\alpha^N = (\alpha^N_1,\ldots, \alpha^N_N)$, $\xi^N=(\xi^N_1, \ldots, \xi^N_N)$, $\N^N = (\N_1^N,\ldots, \N_N^N)$, such that
\begin{equation}
	Law \left( (\alpha^N_1, \xi^N_1, \N_1^N),\ldots, (\alpha^N_N, \xi^N_N, \N_N^N)\right)= (Law (\alpha, \xi, \N))^{\otimes N}.
\label{alphaN}
\end{equation}
\label{not}
\end{notation}

Equation (\ref{alphaN}) says that this control is symmetric. The following is our main result, whose proof is carried out in the next subsection. In addition to (A) and (B), we make the Lipschitz assumptions (A') and (B').

\begin{theorem}
Assume (A') and (B'). Then the vector strategy defined in Notation~\ref{not} is an $\epsilon_N$-Nash equilibrium for the $N$-player game for any $N$ where $\epsilon_N\leq \frac{C}{\sqrt{N}}$ and $C=C(T,d,\nu(U),K_1, K_2)$ is a constant. 
\label{teo1}
\end{theorem}

An analogous result holds when considering feedback strategies, but we state it separately.
\begin{notation}
Let $\left( \left((\Omega, \mathcal{F}, P; (\mathcal{F}_t)_{t\in [0,T]}), \widehat{\gamma}, \xi, \N \right), m, X_{\widehat{\gamma},m}  \right)$
be a relaxed feedback solution of the mean field game (\ref{mfgrf}), which exists assuming (A) and (B) by Theorem \ref{teorf}.
 Fix $N\in\mathbb{N}$ 
and let $\gamma\in\mathbb{A}$ be as in Proposition \ref{epsilon},  satisfying  (\ref{mepsf}) and (\ref{Jepsf}) with $\epsilon = \frac{1}{\sqrt{N}}$.
Then  the tuple $\left((\Omega, \mathcal{F}, P; (\mathcal{F}_t)_{t\in [0,T]}), \gamma^N, \xi^N, \N^N \right)$
denotes the feedback strategy vector where $\xi^N=(\xi^N_1, \ldots, \xi^N_N)$, $\N^N = (\N_1^N,\ldots, \N_N^N)$, $\gamma^N = (\gamma^N_1,\ldots, \gamma^N_N)$ such that
\begin{equation}
\gamma^N_i(t,x^N) := \gamma(t,x^N_i)
\label{gamN}
\end{equation}
for any $t$, $i$ and $x^N=(x^N_1,\ldots, x^N_N) \in \Sigma^N$, and the $(\xi^N_i,\N_i^N)$ are $N$ i.i.d copies of $(\xi,\N)$.
\label{not2}
\end{notation}

Equation (\ref{gamN}) says that this feedback strategy vector is symmetric and decentralized. 
In order to obtain feedback $\epsilon$-Nash equilibria from a mean field game solution, 
we need the  Lipschitz assumptions (A'') and (B'').

\begin{theorem}
Assume (A''), (B''). Then the feedback strategy vector defined in Notation~\ref{not2} is a feedback $\epsilon_N$-Nash equilibrium for the $N$-player game for any $N$ where $\epsilon_N\leq \frac{C}{\sqrt{N}}$ and $C=C(T,d,\nu(U),K_1, K_2)$ is a constant. 

\label{teof}

\end{theorem}

\subsection{Proofs of the results}

In the following $C$ will denote any constant which depends on $T$, $d$, $\nu(U)$ and the Lipschitz constants $K_1$ and $K_2$, but not on $N$, and is allowed to change from line to line. We focus first on open-loop controls.
Fix $N \in\mathbb{N}$ and let the strategy vector $\alpha^N$ be as in Notation \ref{not}. We play this strategy in the $N$-player game:
\begin{equation}
	X^N_i(t)= \xi_i^N + \int_0^t\int_U  f(s, X^N_i(s^-), u, \alpha^N_i(s), \mu_N(s^-)) \N_i^N(ds,du) \qquad i=1,\ldots,N.
\label{X}
\end{equation}
This will be coupled with $Y^N$ defined by
\begin{equation}
	Y^N_i(t)= \xi_i^N + \int_0^t\int_U  f(s, Y^N_i(s^-), u, \alpha^N_i(s), m(s)) \N_i^N(ds,du) \qquad i=1,\ldots,N.
\label{Y}
\end{equation}

Let $\mu^N(t):= \frac{1}{N} \sum_{i=1}^N \delta_{X^N_i (t)}$ be the empirical measure of the system (\ref{X}) and 
$\overline{\mu}_N$ be the empirical measure of (\ref{Y}). Denote $\overline{m}(t) := Law (X_{\alpha, m}(t))$.
By (\ref{meps}) we have
\begin{equation}
	|\overline{m}(t) - m(t)| \leq \frac{1}{\sqrt{N}}
\label{mbar}
\end{equation}
for any $t\geq0$, since $Flow(X_{\rho,m})=m$. 
From (\ref{alphaN}) it follows that
\[
	Law(Y^N_i, \alpha^N_i, \xi_i^N, \N^N_i)= Law (X_{\alpha, m},\alpha, \xi,\N),\quad i\in \{1,\ldots,N\}.
\]
This implies, thanks to Theorem~1 in \cite{fournierguillin15}, that 
\begin{equation}
E |\overline{\mu}^N(t)- \overline{m}(t)| \leq \frac{C}{\sqrt{N}}
\label{mNbar}
\end{equation}
for any $t\in [0,T]$ and $N\in \mathbb{N}$, where $C$ is a constant. This upper bound in $N^{-\frac12}$ cannot be improved, since for these discrete measures a lower bound still in $N^{-\frac12}$ can be found, see again \cite{fournierguillin15}.

\begin{lemma}
Under assumption (A'), for every $t\geq0$ and $i=1,\ldots, N$
\begin{align}
	E |\mu^N(t)- m(t)| &\leq \frac{C}{\sqrt{N}}
	\label{estmu}\\
	 E|X^N_i(t) - Y^N_i(t)| &\leq \frac{C}{\sqrt{N}}.
	\label{esti}
\end{align}

\end{lemma}

\begin{proof}
From (\ref{mbar}) and (\ref{mNbar}) it follows that
\begin{equation}
	E |\overline{\mu}^N(t)- m(t)| \leq \frac{C}{\sqrt{N}}.
\label{mbarNm}
\end{equation}
We estimate $|\mu^N(t) -\overline{\mu}^N(t)|$ using the 1-Wasserstein metric (which is equivalent to the Euclidean metric in $\mathbb{R}^d$)
and (\ref{X}), (\ref{Y}) and the Lipschitz assumption (\ref{lipA}):
\begin{align*}
E|&\mu^N(t) -\overline{\mu}^N(t)| \leq \frac CN \sum_{i=1}^N E|X^N_i(t) - Y^N_i(t)|  \\
&\leq \frac CN \sum_{i=1}^N E|X^N_i(0) - Y^N_i(0)|
+ \frac CN \sum_{i=1}^N\int_0^t\int_U E | f(s, X^N_i(s), u, \alpha^N_i(s), \mu_N(s^-))\\
&- f(s, Y^N_i(s), u, \alpha^N_i(s), m(s)) | \nu(du) ds \\
&\leq \frac CN \sum_{i=1}^N K_1 \int_0^t \left[ E|X^N_i(s) - Y^N_i(s)| + E |\mu^N(s)- m(s)|\right]  ds .
\end{align*}
Hence applying (\ref{mbarNm})
\begin{align*}
	& E |\mu^N(t)- m(t)| + \frac CN \sum_{i=1}^N E|X^N_i(t) - Y^N_i(t)| \\
	&\leq E |\overline{\mu}^N(t)- m(t)| + \frac{2C}{N} \sum_{i=1}^N E|X^N_i(t) - Y^N_i(t)|\\
	&\leq \frac{C}{\sqrt{N}} + 2K_1 \frac1N \sum_{i=1}^N \int_0^t \left[ E|X^N_i(s) - Y^N_i(s)| + E |\mu^N(s)- m(s)|\right]  ds .
\end{align*}
Then we obtain, by Gronwall's lemma, 
\[
	E |\mu^N(t)- m(t)| + \frac CN \sum_{i=1}^N E|X^N_i(t) - Y^N_i(t)| \leq 
\frac{C}{\sqrt{N}} +2K_1 \int_0^t e^{2K(t-s)} \frac{C}{\sqrt{N}} ds \leq \frac{C}{\sqrt{N}}.
\]

Similarly we show (\ref{esti}): using (\ref{X}), (\ref{Y}) and (\ref{estmu}) we get, for any $i$, 
\begin{align*}
	&E|X^N_i(t) - Y^N_i(t)|\\
	&\leq \int_0^t\int_U E \left| f(s, X^N_i(s), u, \alpha^N_i(s), \mu_N(s)) - f(s, Y^N_i(s), u, \alpha^N_i(s), m(s)) \right| \nu(du) ds \\
	&\leq K_1 \int_0^t \left[ E|X^N_i(s) - Y^N_i(s)| + E |\mu^N(s)- m(s)|\right]  ds\\
	&\leq K_1 \int_0^t \left[ E|X^N_i(s) - Y^N_i(s)| + \frac{C}{\sqrt{N}} \right]  ds
\end{align*}
and hence  
$E|X^N_i(t) - Y^N_i(t)| \leq \frac{C}{\sqrt{N}}$ by  Gronwall's lemma.
\end{proof}

We are now in the position to state the result about the costs. 
Because of the symmetry of the problem, for the prelimit we shall consider only player one ($i=1$).

\begin{lemma}
Under assumptions (A') and (B')
\begin{equation}
|J^N_1(\alpha^N)-J(\rho,m)| \leq \frac{C}{\sqrt{N}}.
\label{JNrho}
\end{equation}
\label{lem1}
\end{lemma}

\begin{proof}
Inequality (\ref{Jeps}), together with notation \ref{not}, yields
\begin{equation}
|J(\alpha,m)-J(\rho,m)| \leq \frac{C}{\sqrt{N}}.
\label{JepsN}
\end{equation}
While  from (\ref{lipB}), (\ref{estmu}) and (\ref{esti}) we have
\begin{align*}
|J^N_1(\alpha^N)&-J(\alpha,m)| \leq E |\psi(X^N_1(T), \mu^N(T)) - \psi(Y^N_1(T), m(T))| \\
&\qquad+ E\int_0^T |c(t,X^N_1(t), \alpha^N_1(t), \mu^N(t)) - c(t,Y^N_1(t), \alpha^N_1(t), m(t))|dt\\
&\leq K_2 \int_0^T \left[ E|X^N_1(t) - Y^N_1(t)| + E |\mu^N(t)- m(t)|\right] dt \\
&\qquad + K_2 \left[E|X^N_1(T) - Y^N_1(T)| + E |\mu^N(T)- m(T)|\right] \\
&\leq K_2 T \frac{C}{\sqrt{N}} + K_2 \frac{C}{\sqrt{N}} \leq \frac{C}{\sqrt{N}},
\end{align*}
which, combined with (\ref{JepsN}), gives the claim.
\end{proof}

We consider then any $\beta\in\mathcal{A}$ and the perturbed strategy vector $[\alpha^{N,-1},\beta]$. We denote by 
$\widetilde{X}^N$ the solution to
\begin{equation}
\widetilde{X}^N_i(t)= \xi_i^N + \int_0^t\int_U  f(s, \widetilde{X}^N_i(s^-), u, [\alpha^{N,-1},\beta]_i(s), \widetilde{\mu}_N(s^-)) 
\N_i^N(ds,du) 
\label{Xtilde}
\end{equation}
for each $i=1,\ldots,N$. Set also $\widetilde{Y}_1 := X_{\beta,m}$ and 
$\widetilde{\mu}_N(t) := \frac1N \sum_{i=1}^N \delta_{\widetilde{X}^N_i(t)}.$

\begin{lemma}
Under assumption (A'), for any  $t\geq0$  and $\beta\in\mathcal{A}$
\begin{eqnarray}
E |\mu^N(t) - \widetilde{\mu}^N(t)| &\leq& \frac CN
\label{est1}\\
E |\widetilde{\mu}^N(t)- m(t)| &\leq& \frac{C}{\sqrt{N}}
	\label{est2}\\
	 E|\widetilde{X}^N_1(t) - \widetilde{Y}_1(t)| &\leq& \frac{C}{\sqrt{N}}.
	\label{est3}
\end{eqnarray}
\label{lem10}
\end{lemma}

\begin{proof}
We make the rough estimate
\begin{align*}
	& E|\mu^N(t) -\widetilde{\mu}^N(t)| \leq  \frac1N E|X^N_1(t) - \widetilde{X}^N_1(t)| + \frac1N \sum_{i=2}^N E|X^N_i(t) - \widetilde{X}^N_i(t)| \\
	&\leq \frac dN + \frac1N \sum_{i=2}^N\int_0^t\int_U E | f(s, X^N_i(s), u, \alpha^N_i(s), \mu_N(s))\\
	&\qquad \qquad- f(s, \widetilde{X}^N_i(s), u, \alpha^N_i(s) ,\widetilde{\mu}_N(s)) | \nu(du) ds \\
	&\leq \frac dN + \frac1N \sum_{i=2}^N K_1 \int_0^t \left[ E|X^N_i(s) -  \widetilde{X}^N_i(s)| + E |\mu^N(s)- \widetilde{\mu}_N(s)|\right]  ds .
\end{align*}
Hence 
\begin{align*}
E |\mu^N(t)&- \widetilde{\mu}_N(t)| + \frac1N \sum_{i=2}^N E|X^N_i(t) - \widetilde{X}^N_i(t)| \\
&\leq
\frac dN + 2 K_1 \frac1N \sum_{i=2}^N  \int_0^t \left[ E|X^N_i(s) -  \widetilde{X}^N_i(s)| + E |\mu^N(s)- \widetilde{\mu}_N(s)|\right]  ds 
\end{align*}
and then, by Gronwall's lemma, 
\[
	E |\mu^N(t)- \widetilde{\mu}_N(t)| + \frac1N \sum_{i=2}^N E|X^N_i(t) - \widetilde{X}^N_i(t)| \leq \frac dN e^{2K_1 T} \leq \frac CN.
\]
Therefore (\ref{est1}) is proved. Estimate (\ref{est2}) follows from (\ref{est1}) and (\ref{estmu}) and the fact that $\frac1N \leq \frac{1}{\sqrt{N}}$ for any $N\in\mathbb{N}$. While (\ref{est3}) is a consequence of (\ref{est2}):
\begin{align*}
E|&\widetilde{X}^N_1(t) - \widetilde{Y}_1(t)| \\
&\leq\int_0^t\int_U E \left| f(s, \widetilde{X}^N_1(s), u, \beta(s), \widetilde{\mu}_N(s))
- f(s,\widetilde{Y}_1 (s), u, \beta(s), m(s)) \right| \nu(du) ds \\
&\leq K_1 \int_0^t \left[ E|\widetilde{X}^N_1(s) - \widetilde{Y}_1(s)|+ E |\mu^N(s)- m(s)|\right]  ds \\
&\leq
K_1 \int_0^t \left[ E|\widetilde{X}^N_1(s) - \widetilde{Y}_1(s)| + \frac{C}{\sqrt{N}} \right]  ds
\end{align*}
and we conclude by  Gronwall's lemma.
\end{proof}

\begin{lemma}
Under assumptions (A') and (B')
\begin{equation}
	|J^N_1([\alpha^{N,-1},\beta])-J(\beta,m)| \leq \frac{C}{\sqrt{N}}.
\label{JNbeta}
\end{equation}
\label{lem2}
\end{lemma}

\begin{proof}
Inequalities (\ref{lipB}), (\ref{est2}) and (\ref{est3}) give
\begin{align*}
|J^N_1([\alpha^{N,-1},&\beta])-J(\beta,m)| \leq 
E |\psi(\widetilde{X}^N_1(T), \widetilde{\mu}^N(T)) - \psi(\widetilde{Y}_1(T), m(T))|\\ 
&\qquad + E\int_0^T |c(t,\widetilde{X}^N_1(t), \beta(t), \widetilde{\mu}^N(t)) - c(t,\widetilde{Y}_1(t), \beta(t), m(t))|dt\\
\leq& K_2 \int_0^T \left[ E|\widetilde{X}^N_1(t) - \widetilde{Y}_1(t)| + E |\widetilde{\mu}^N(t)- m(t)|\right] dt \\
&\qquad+ K_2 \left[E|\widetilde{X}^N_1(T) - \widetilde{Y}_1(T)| + E |\widetilde{\mu}^N(T)- m(T)|\right] \\
\leq& K_2 T \frac{C}{\sqrt{N}} + K_2 \frac{C}{\sqrt{N}} \leq \frac{C}{\sqrt{N}}.
\end{align*}

\end{proof}

Theorem \ref{teo1} is now a consequence of Lemmata \ref{lem1} and \ref{lem2}:

\begin{proof}[Proof of Theorem \ref{teo1}]

Inequalities (\ref{JNrho}), (\ref{JNbeta}), and the optimality of $\rho$ yield 
\[
	J^N_1(\alpha^N) \leq J(\rho,m) + \frac{C}{\sqrt{N}}\leq J(\beta,m) + \frac{C}{\sqrt{N}} 
\leq J^N_1([\alpha^{N,-1},\beta]) + \frac{C}{\sqrt{N}}.
\]
\end{proof}

\begin{remark}
We observe that $\alpha^N$ is still an $\epsilon_N$-Nash equilibrium if we assume only (B) instead of (B'), but without the estimate of the order of convergence $\epsilon_N \leq \frac{C}{\sqrt{N}}$. Namely, there exists a sequence $(\epsilon_N)$ such that $\lim_{N\rightarrow\infty} \epsilon_N =0$.
\end{remark}

\begin{proof}[Proof of Theorem \ref{teof}]

The argument is the same as in the proof of Theorem \ref{teo1}. The difference is that equations (\ref{X}), (\ref{Y}) 
and (\ref{Xtilde}) become respectively, for each $ i=1,\ldots,N$,
\begin{align*}
	X^N_i(t) &= \xi_i^N + \int_0^t\int_U  f(s, X^N_i(s^-), u, \gamma(s, X^N_i(s^-)), \mu_N(s^-)) \N_i^N(ds,du), \\
	Y^N_i(t) &= \xi_i^N + \int_0^t\int_U  f(s, Y^N_i(s^-), u, \gamma(s, Y^N_i(s^-)), m(s)) \N_i^N(ds,du) \\
\intertext{and}
	\widetilde{X}^N_i(t) &= \xi_i^N + \int_0^t\int_U  f(s, \widetilde{X}^N_i(s^-), u, [\gamma^{N,-1},\beta]_i(s), \widetilde{\mu}_N(s^-)) \N_i^N(ds,du),
\end{align*}
where the latter means that 
\[
	\widetilde{X}^N_1(t)= \xi_1^N + \int_0^t\int_U  f(s, \widetilde{X}^N_1(s^-), u, \beta(s), \widetilde{\mu}_N(s^-)) \N_1^N(ds,du)
\]
and
\[
	\widetilde{X}^N_i(t)= \xi_i^N + \int_0^t\int_U  f(s, \widetilde{X}^N_i(s^-), u, \gamma(s, \widetilde{X}^N_i(s^-)), \widetilde{\mu}_N(s^-)) 
\N_i^N(ds,du)
\]
for $ i=2,\ldots,N$, thanks to Notation \ref{n1}.
The estimates we need to apply Gronwall's lemma, in particular in the proof of Lemma~\ref{lem10}, are found using also (\ref{lipC}) and the fact that 
$dist(\gamma(s,x), \gamma(s,y))  \leq Diam(A) |x-y|$ for every $s$ and each $x$ and $y$ in the finite $\Sigma$.
\end{proof}

\section{Conclusions}

We summarize here the results we have obtained. The assumptions are given in Section~2.1 and verified for a natural shape of the dynamics in Lemmata \ref{lemf1} and \ref{ffeed}.

\begin{enumerate}
	\item Under assumptions (A) and (B), there exist a relaxed mean field game solution and a relaxed feedback mean field game solution (in the sense of Definition~\ref{solrel}), see Theorems \ref{teo} and \ref{teorf}, respectively.
	
	\item Assuming (A), (B) and (C),
	there exists a feedback solution of the mean field game (Definition \ref{sol}), see Corollary \ref{coro}. The feedback mean field game solution is unique for small $T$ under the additional assumptions of Section~4.3 by Theorem \ref{uniq}; uniqueness for arbitrary time horizon holds under the Lasry-Lions monotonicity assumptions, see Theorem~\ref{ThMonotonicity}.

	\item The relaxed mean field game solutions provide $\epsilon_N$-Nash equilibria for the $N$-player game (cf.\ Definition \ref{eqnash}), both in open-loop and in feedback form (not relaxed), with $\epsilon_N\leq \frac{C}{\sqrt{N}}$ . If (A') and (B') hold, then the symmetric open-loop strategy vector defined in Notation~\ref{not} is an $\epsilon_N$-Nash equilibrium by Theorem \ref{teo1}. Assuming (A'') and (B''), the feedback strategy vector defined in Notation \ref{not2}, which is symmetric and decentralized, is a feedback $\epsilon_N$-Nash equilibrium thanks to Theorem \ref{teof}.
\end{enumerate}

\appendix

\section{Relaxed Poisson measures}

In order to state the definition of the relaxed Poisson random measure we first need to define the canonical space of integer valued random measures on a metric space $E$. Following \cite{jacod79}, the setting is:
\begin{itemize}
	\item $\overline{\Omega}$ is the set of sequences  $(t_n,y_n) \subset [0, +\infty] \times E$ such that $(t_n)$ is increasing and 
	$t_n < t_{n+1}$ if $t_n < +\infty$; set $t_0:=0$ and $t_\infty:= \lim_n t_n$;
	\item if $\overline{\omega} = (t_n,y_n)_{n\in \mathbb{N}}$ write $T_n(\overline{\omega}):=t_n$ and $Y_n(\overline{\omega}):=y_n$;
	\item the \emph{canonical random measure} is
	\[
	\overline{\N}(\overline{\omega}, B) := \sum_{n\in \mathbb{N}} \mathbbm{1}_{\left\{T_n(\overline{\omega})<\infty\right\}}
	\delta_{(T_n(\overline{\omega}), Y_n(\overline{\omega})} (B)
	\]
	for any $B\in\mathcal{B}([0, +\infty[ \times E )$;
	\item $\overline{\mathcal{G}}_t := \sigma \left(\overline{\N}(\cdot, B) : B\in\mathcal{B}([0, t] \times E )\right)$, 
	$\overline{\mathcal{F}}_0$ is given, 
	$\overline{\mathcal{F}}_t = \overline{\mathcal{F}_0} \vee \left(\cap_{s<t} \overline{\mathcal{G}}_s\right)$, 
	$ \overline{\mathcal{F}}=\overline{\mathcal{F}}_\infty$ and $\overline{\mathbb{F}}= (\overline{\mathcal{F}}_t)_{t\geq0}$.
\end{itemize}
The filtered space $(\overline{\Omega},\overline{\mathcal{F}},\overline{\mathbb{F}})$ is then the \emph{canonical space of integer valued random measures on $E$}. A probability measure on it is the law of an integer valued random measure on $E$, given an initial condition on 
$\overline{\mathcal{F}_0}$. 
Note that the canonical measure $\overline{\N}$ is not the identity: for this reason we can work with $\mathcal{M} = \mathcal{M}([0, +\infty[ \times E)$ as the state space of a random measure. Moreover, the set of integer valued random measures is vaguely closed in $\mathcal{M}$:
see Theorem 15.7.4 in \cite{kallenberg86} and the references therein.

Let now $\Theta$ be any integer valued random measure defined on a filtered probability space $(\Omega,\mathcal{F},\mathbb{F},P)$. It is determined by a sequence of stopping times $T_n$ and random variables $X_n$ which are $\mathcal{F}_{T_n}$-measurable. To any $\Theta$ is associated its \emph{compensator}, that is, a positive random measure $\eta$  on $E$ such that
\begin{enumerate}
	\item $\eta([0,t]\times B)_{t\geq0}$ is predictable for any $B\in\mathcal{B}(E)$;
	\item $(\Theta([0,t\wedge T_n]\times B) -\eta([0,t\wedge T_n]\times B))_{t\geq0}$ is an $\mathbb{F}$-martingale for each $n$ and $B$;
	\item $\eta(\left\{t\right\}\times E)\leq 1$ for each $t$ and $\eta([T_\infty, \infty[ \times E)=0$.
\end{enumerate}
The compensator exists and is unique (up to a modification on a $P$-null set) for any $\Theta$. The proof can be found in \citet{jacod75}, where the author also shows that a process with the above properties uniquely determines an integer valued random measure.

Consider then an arbitrary measurable space $(\Omega',\mathcal{F}')$ and define $\Omega := \overline{\Omega} \times \Omega'$.
Set $\overline{\mathcal{F}_0} := \left\{\varnothing,\overline{\Omega}\right\}$ and
$\mathcal{F}_0 :=\overline{\mathcal{F}_0}\otimes\mathcal{F}'$.
The canonical random measure $\overline{\N}$ on $\overline{\Omega}$ is extended to $\Omega$ via 
$(T_n,Y_n).(\overline{\omega},\omega'):=(T_n,Y_n).(\overline{\omega})$. Set
$\mathcal{F}_t:= \overline{\mathcal{F}}_t \vee \mathcal{F}_0$.

\begin{theorem}[\citet{jacod75}]
Let $P_0$ be a probability measure on $(\Omega, \mathcal{F}_0)$ and $\eta$ a predictable random measure satisfying (1) and (3). Then there exists a unique probability measure $P$ on $(\Omega,\mathcal{F}_\infty)$ whose restriction to 
$\mathcal{F}_0$ is $P_0$ and for which $\eta$ is the compensator of $\overline{\N}$.
\label{Jac}
\end{theorem}

By means of this theorem, we are able to define properly a relaxed Poisson measure.
Consider a relaxed control $((\Omega'', \mathcal{F}'', P''; \mathbb{F}''), \rho, \xi, \N)  \in\mathcal{R}$
and let  $\Omega'= \mathcal{D}\times \Sigma \times \overline{\Omega}$ be the state space of the process $\rho$, the initial distribution $\xi$ and the Poisson random measure $\N$. 
The $\sigma$-algebra $\mathcal{F}'$ is generated by the processes and $P_0$ is the joint law of $(\rho, \xi, \N)$.
So a relaxed Poisson measure $\N_\rho$, related to the relaxed control $\rho$, is an integer valued random measure on $[0,T]\times U \times A$ whose compensator $\eta$, calculated on $[0,t]$, $U_0$, $A_0$, is $\nu(U_0) \rho([0,t]\times A_0)$. Its law is uniquely determined on $\overline{\Omega}$ and thus has the martingale properties (\ref{mart}) and (\ref{M}). 
Moreover, the joint law of $(\N_\rho,\rho,\xi,\N)$ is uniquely determined.

We can give an explicit construction of $\N_\rho$. Let $\rho\in\mathcal{R}$ and $(\alpha_n)$ be a sequence in $\mathcal{A}$ which tends to $\rho$ in the sense of Lemma~\ref{chat}, the chattering lemma. Denote by $\rho^{\alpha_n}$ the relaxed control representation of $\alpha_n$ and 
construct $\N_{\alpha_n}$ as in (\ref{Nal}):
$\N_{\alpha_n}(t,U_0,A_0) := \int_0^t \int_{U_0}\mathbbm{1}_{A_0}(\alpha_n(s))  \N(ds,du)$.
Then, by Theorem~\ref{tight}, the sequence $(X_{\alpha_n}, \rho^{\alpha_n},\N_{\alpha_n})$ is tight and any subsequence converges in distribution to $(X_{\rho}, \rho,\N_{\rho})$. The marginals are uniquely defined in this way, while to show that the joint law of
$(\rho,\N_{\rho})$ is unique we need to invoke the above Theorem~\ref{Jac}.

\subsection{Proof of Lemma~\ref{lemmaexi} }

Let $m\in\mathcal{L}$ be fixed, which we shall omit. Let $\mathcal{Z}$ be the space of stochastic processes 
with paths in $D([0,T],\Sigma)$ and equip it with the norm 
$||X||= E\left[\sup_{0\leq t \leq T} |X(t)|\right]$. Let $\rho \in\mathcal{R}$ and define the map 
$G:\mathcal{Z} \longrightarrow \mathcal{Z}$ by 
\[
	G_t(X):= \xi + \int_0^t\int_U\int_A  f(s,X(s^-), u, a) \N_\rho(ds,du,da)
\]
for any $X\in \mathcal{Z}$. If we prove that this map is a contraction in the norm $||\cdot||$, then pathwise existence and uniqueness of  solutions to equation 
(\ref{mfgr}) follow. We have, for any $X, Y \in \mathcal{Z}$,
\[
	|G_t(X)-G_t(Y)| \leq \int_0^t\int_U\int_A | f(s,X(s^-), u, a) - f(s,Y(s^-), u, a)|\N_\rho(ds,du,da),
\]
hence 
\begin{align*}
	&E\left[\sup_{0\leq t\leq T} |G_t(X)-G_t(Y)|\right]\\
	&\qquad\leq E\int_0^T\int_U\int_A | f(s,X(s), u, a) - f(s,Y(s), u, a)|\rho_s(da) \nu(du) ds\\
	&\qquad\leq  K_1 E \int_0^T \int_A |X(s) -Y(s)| \rho_s(da) ds \leq K_1 T E \left[\sup_{0\leq t\leq T} |X(s) -Y(s)|\right]
\end{align*}
thanks to (\ref{lip}) and the fact that $\rho_s$ is a probability measure. Therefore $G$ is a contraction if $T <\frac{1}{K_1}$, and so uniqueness is proved for small time horizon; but then iterating the same argument, we have uniqueness for any $T$.

Consider now $\widehat{\gamma}\in \widehat{\mathbb{A}}$ and define
$\widehat{G}:\mathcal{Z} \longrightarrow \mathcal{Z}$ by
\[
	\widehat{G}_t(X):= \xi + \int_0^t\int_U\int_A  f(s,X(s^-), u, a) \N_{\rho^{\widehat{\gamma},X}}(ds,du,da)
\]
for any process $X\in\mathcal{Z}$. Then for any $X$ and $Y$ we have
$||\widehat{G}(X) - \widehat{G}(Y)|| \leq ||Z_1|| + ||Z_2||$
where
\[
	Z_1(t) := \int_0^t\int_U\int_A  |f(s,X(s^-), u, a)- f(s,Y(s^-), u, a)| \N_{\rho^{\widehat{\gamma},Y}}(ds,du,da)
\]
and
\[
	Z_2(t) := \int_0^t\int_U\int_A  |f(s,X(s^-), u, a)| 
\left| \N_{\rho^{\widehat{\gamma},X}}-\N_{\rho^{\widehat{\gamma},Y}}\right|(ds,du,da),
\]
where $|\Theta|$ denotes the total variation of the signed measure $\Theta$ defined for any $C\in \mathcal{B}([0,T]\times U \times A)$
by $|\Theta|(C) := \sup_{E\subseteq C} |\Theta(E)|$; while the total variation norm is 
$||\Theta||_{TV} = |\Theta|([0,T]\times U \times A)$.
The first term $Z_1$ is bounded as above yielding
$||Z_1||\leq K_1 T||X-Y||$. For the second term, we use $|f|\leq d$ to obtain
\[
	\sup_{0\leq t\leq T} Z_2(t) \leq d ||\N_{\rho^{\widehat{\gamma},X}}-\N_{\rho^{\widehat{\gamma},Y}}||_{TV}
= d \sup_{E\subset [0,T]\times U \times A} 
\left| \N_{\rho^{\widehat{\gamma},X}}(E)-\N_{\rho^{\widehat{\gamma},Y}}(E)\right|.
\]
Thanks to (\ref{mart}) and (\ref{relcon}), we have $E||\N_{\rho^{\widehat{\gamma},X}}-\N_{\rho^{\widehat{\gamma},Y}}||_{TV} \leq 2T\nu(U)$,
saying that the right-hand side above is finite $P$-a.s. 
Since the measure $\N_{\rho^{\widehat{\gamma},X}}-\N_{\rho^{\widehat{\gamma},Y}}$ is integer valued, we can assume that the above supremum is attained on a set $C(\omega)$ for $P$-a.e. $\omega$, giving thus a random set $C$. Moreover, we may assume that on such a set the random measure considered is positive.
The martingale property (\ref{M}) now gives
\begin{align*}
	||Z_2|| &\leq d E \left[ \N_{\rho^{\widehat{\gamma},X}}(C) -\N_{\rho^{\widehat{\gamma},Y}}(C) \right]\\
	&=d \left|E \int_0^T \int_U\int_A \mathbbm{1}_{C}(t,u,a)[\widehat{\gamma}(t,X(t)) - \widehat{\gamma}(t,Y(t))](da)\nu(du)dt\right|\\
	&\leq d E \int_0^T |\widehat{\gamma}(t,X(t)) - \widehat{\gamma}(t,Y(t))|(A) \nu(U) dt\\
	&\leq  2 \nu(U) d E \int_0^T |X(t) -Y(t)| dt \leq K_1 T ||X-Y||,
\end{align*}
where in the last line above we have used the fact that $\widehat{\gamma}$ is a probability measure and $|x-y| \geq 1$ for each $x\neq y\in \Sigma$.
Therefore, for $T<\frac{1}{2 K_1}$, the map $\widehat{G}$ is a contraction; the claim follows iterating the above procedure.

\end{document}